\documentclass[onefignum,onetabnum,oneeqnum,onethmnum]{siamart190516}

\usepackage{amsmath}
\usepackage{lipsum}
\usepackage{amsfonts}
\usepackage{fullpage}
\usepackage{amssymb}
\usepackage{graphicx}
\usepackage{epstopdf}
\usepackage{algorithmic}
\usepackage{paralist}
\usepackage{mathtools}
\ifpdf
  \DeclareGraphicsExtensions{.eps,.pdf,.png,.jpg}
\else
  \DeclareGraphicsExtensions{.eps}
\fi

\newsiamremark{remark}{Remark}
\newsiamremark{hypothesis}{Hypothesis}
\newsiamthm{claim}{Claim}

\title{Shape optimization of peristaltic pumps transporting rigid particles in Stokes flow
}

\author{Marc Bonnet\thanks{POEMS (CNRS, INRIA, ENSTA), ENSTA Paris, 91120 Palaiseau, France
  (\email{mbonnet@ensta.fr}).}
\and Ruowen Liu\thanks{Department of Mathematics, University of Michigan, Ann Arbor, United States
  (\email{ruowen@umich.edu}, \email{shravan@umich.edu}, \email{hszhu@umich.edu}).}
\and Shravan Veerapaneni\footnotemark[2]
\and Hai Zhu\footnotemark[2]}

\usepackage{amsopn}
\usepackage{autonum}

\newlength{\kaka}
\newcommand{\Asub}{_{\text{\scriptsize A}}}
\newcommand{\bfa}{\boldsymbol{a}}

\newcommand{\bfxh} {\hat{\bfx}}
\newcommand{\bfc} {\boldsymbol{c}}
\newcommand{\bfd} {\boldsymbol{d}}
\newcommand{\bfD} {\boldsymbol{D}}
\newcommand{\bfe} {\boldsymbol{e}}
\newcommand{\bfE} {\boldsymbol{E}}
\newcommand{\bff} {\boldsymbol{f}}

\newcommand{\bffh} {\hat{\boldsymbol{f}}{}}
\newcommand{\bfg} {\boldsymbol{g}}

\newcommand{\bfh} {\boldsymbol{h}}
\newcommand{\bfhh} {\hat{\boldsymbol{h}}{}}
\newcommand{\bfI} {\boldsymbol{I}}
\newcommand{\bfk} {\boldsymbol{k}}
\newcommand{\bfn} {\boldsymbol{n}}
\newcommand{\bfna}{\boldsymbol{\nabla}}

\newcommand{\bfR} {\boldsymbol{R}}
\newcommand{\bfr}{\boldsymbol{r}}
\newcommand{\bfrh}{\boldsymbol{\rho}}
\newcommand{\bfrhh}{\hat{\bfrh}}
\newcommand{\bfsig}{\boldsymbol{\sigma}}
\newcommand{\bftau}{\boldsymbol{\tau}}
\newcommand{\bfth}{\boldsymbol{\theta}}
\newcommand{\bfu} {\boldsymbol{u}}
\newcommand{\bfU} {\boldsymbol{U}}
\newcommand{\bfuD} {\bfu\Dsup}

\newcommand{\bfuh} {\hat{\boldsymbol{u}}{}}
\newcommand{\bfv} {\boldsymbol{v}}
\newcommand{\bfw} {\boldsymbol{w}}
\newcommand{\bfps} {\boldsymbol{\psi}}
\newcommand{\bfdel} {\boldsymbol{\delta}}
\newcommand{\bfxr} {\mathbf{x}}
\newcommand{\bfx} {\boldsymbol{x}}
\newcommand{\bfy} {\boldsymbol{y}}
\newcommand{\bfxi}{\boldsymbol{\xi}}
\newcommand{\bfz} {\boldsymbol{z}}
\newcommand{\bfze}{\mathbf{0}}
\newcommand{\bfzet}{\boldsymbol{\zeta}}

\newcommand{\per}{_{\text{\scriptsize per}}}
\newcommand{\dd}[1]{\protect\settowidth{\kaka}{\text{$#1$}}\protect\makebox[\kaka]{\ensuremath{\buildrel{\scriptstyle\star}\over{#1}}}}
\newcommand{\bdot}[1]{\protect\settowidth{\kaka}{\text{$#1$}}\protect\makebox[\kaka]{\ensuremath{\buildrel{\Bdot}\over{#1}}}}
\newcommand{\Bdot}{\scalebox{0.5}{$\;\bullet$}}
\newcommand{\Dcal}{\mathcal{D}}

\newcommand{\del}[1]{\partial_{#1}}
\newcommand{\der}[2]{\dfrac{\text{d}#1}{\text{d}#2}}
\newcommand{\dip} {\! :\!}
\newcommand{\Div}{\mbox{div}\,}
\newcommand{\DivS} {\mbox{div}_{\! S}}
\newcommand{\dO}{\partial\Omega}
\newcommand{\ds}{\,\text{d}s}
\newcommand{\dt}{\,\text{d}t}
\newcommand{\Dsub}{_{\text{\scriptsize D}}}
\newcommand{\Dsup}{^{\text{\tiny D}}}
\newcommand{\dV}{\;\text{d}V}

\newcommand{\fhat}{\hat{f}}
\newcommand{\FS}{\ensuremath{\mbox{\boldmath $\mathcal{F}$}}}
\newcommand{\HS}{\ensuremath{\mbox{\boldmath $\mathcal{H}$}}}
\newcommand{\G}{\Gamma}
\newcommand{\GL}{\Gamma_L}
\newcommand{\Gp}{\G_{p}}
\newcommand{\Gpm}{\G_L}
\newcommand{\Gpp}{\G_0}
\newcommand{\Gsup}{^{\text{\tiny G}}}
\newcommand{\hath}{\hat{h}}
\newcommand{\hatp}{\hat{p}}
\newcommand{\hatr}{\hat{r}}

\newcommand{\ioT}{\int_{\oo(T)}}

\newcommand{\idO}{\int_{\dO}}
\newcommand{\ig}{\int_{\!\gamma}}
\newcommand{\igt}{\int_{\!\gamma(t)}}

\newcommand{\iG} {\int_{\G}}

\newcommand{\iGL} {\int_{\GL}}
\newcommand{\inv}[1]{\dfrac{1}{#1}}
\newcommand{\pinv}[1]{\tfrac{1}{#1}}
\newcommand{\iO}{\int_{\OO}}
\newcommand{\iOt}{\int_{\OO(t)}}
\newcommand{\iS}{\int_{S}}
\newcommand{\iT}{\int_{0}^T}

\newcommand{\lbra}{\big\langle\hspace*{0.1em}}
\newcommand{\Lcal}{\mathcal{L}}
\newcommand{\lcb}{\big\{\hspace*{0.1em}}
\newcommand{\Lcb}{\Big\{\hspace*{0.1em}}
\newcommand{\lpar}{\big(\hspace*{0.1em}}
\newcommand{\Lpar}{\Big(\,}
\newcommand{\lsqb}{\big[\hspace*{0.1em}}
\newcommand{\Lsqb}{\Big[\hspace*{0.1em}}

\newcommand{\Ocal}{\mathcal{O}}
\newcommand{\OO}{\Omega}
\newcommand{\oo}{\omega}
\newcommand{\ooh}{\hat{\varrho}}
\newcommand{\Pcal}{\mathcal{P}}

\newcommand{\Wsub}{_{\text{\tiny W}}}

\newcommand{\Rabs}{\hspace*{0.1em}\Big|}
\newcommand{\Rbb} {\mathbb{R}}
\newcommand{\rbra}{\hspace*{0.1em}\big\rangle}
\newcommand{\rcb}{\hspace*{0.1em}\big\}}
\newcommand{\Rcb}{\hspace*{0.1em}\Big\}}
\newcommand{\rpar}{\hspace*{0.1em}\big)}
\newcommand{\rmx}{\mathrm{x}}
\newcommand{\Rpar}{\,\Big)}
\newcommand{\RS}{\text{\boldmath $\mathcal{R}$}}

\newcommand{\rsqb}{\hspace*{0.1em}\big]}
\newcommand{\Rsqb}{\hspace*{0.1em}\Big]}
\newcommand{\Scal}{\mathcal{S}}
\newcommand{\shcap}{\hspace*{-0.1em}\cap\hspace*{-0.1em}}
\newcommand{\shcup}{\hspace*{-0.1em}\cup\hspace*{-0.1em}}
\newcommand{\shdeq}{\hspace*{-0.1em}:=\hspace*{-0.1em}}
\newcommand{\sheq}{\hspace*{-0.1em}=\hspace*{-0.1em}}
\newcommand{\shg}{\hspace*{-0.1em}>\hspace*{-0.1em}}

\newcommand{\shin}{\hspace*{-0.1em}\in\hspace*{-0.1em}}
\newcommand{\shl}{\hspace*{-0.1em}<\hspace*{-0.1em}}
\newcommand{\shleq}{\hspace*{-0.1em}\leq\hspace*{-0.1em}}
\newcommand{\shm}{\hspace*{-0.1em}-\hspace*{-0.1em}}
\newcommand{\shp}{\hspace*{-0.1em}+\hspace*{-0.1em}}
\newcommand{\shsubs}{\hspace*{-0.1em}\subset\hspace*{-0.1em}}
\newcommand{\shtimes}{\hspace*{-0.1em}\times\hspace*{-0.1em}}

\newcommand{\sip} {\! \cdot\!}
\newcommand{\Ssub}{_{\text{\scriptsize S}}}
\newcommand{\suite}[1][0ex]{\notag \\[#1] & \mbox{}\hspace{15pt}}
\newcommand{\Tcal}{\mathcal{T}}
\newcommand{\tdemi} {\tfrac{1}{2}}
\newcommand{\tensor}{\hspace*{-1pt}\otimes\hspace*{-1pt}}
\newcommand{\Tsup}{^{\text{\scriptsize T}}}
\newcommand{\US}{\text{\boldmath $\mathcal{U}$}}
\newcommand{\Vsub}{_{\text{\scriptsize V}}}

\DeclareMathOperator*{\argmin}{arg\;min}

\begin{document}

\maketitle

\begin{abstract}
This paper presents a computational approach for finding the optimal shapes of peristaltic pumps transporting rigid particles in Stokes flow. In particular, we consider shapes that minimize the rate of energy dissipation while pumping a prescribed volume of fluid, number of particles and/or distance traversed by the particles over a set time period. Our approach relies on a recently developed fast and accurate boundary integral solver for simulating multiphase flows through periodic geometries of arbitrary shapes. In order to fully capitalize on the dimensionality reduction feature of the boundary integral methods, shape sensitivities must ideally involve evaluating the physical variables on the particle or pump boundaries only. We show that this can indeed be  accomplished owing to the linearity of Stokes flow. The forward problem solves for the particle motion in a slip-driven pipe flow while the adjoint problems in our construction solve quasi-static Dirichlet boundary value problems backwards in time, retracing the particle evolution. The shape sensitivies simply depend on the solution of one forward and one adjoint (for each shape functional) problems. We validate these analytic shape derivative formulas by comparing against finite-difference based gradients and present several examples showcasing optimal pump shapes under various constraints.    
\end{abstract}



\section{Introduction}

Transporting rigid and deformable particles suspended in a viscous fluid with precise control is a challenging but crucial task in microfluidics \cite{kirby2010micro}. A classical engineering approach---one that is commonly found in biological systems (e.g., see \cite{jaffrin1971peristaltic, fauci2006biofluidmechanics, bornhorst2017gastric})---is the use of periodic contraction waves of the enclosing tube to drive the particulate flows. This mechanism is known as peristalsis. Computationally, the forward problem of simulating the particle transport for a given peristaltic wave shape has been considered in a number of works; a few recent ones that consider various physical scenarios include \cite{takagi2011peristaltic, chrispell2011peristaltic, aranda2015model, marconati2017transient, poursharifi2018peristaltic}. However, the inverse problem of finding the optimal wave shapes (e.g., that minimize the pump’s power loss) received little attention, primarily owing to the computational challenges associated with its solution---every shape iteration requires time-dependent solution of a rigid (or deformable) particle motion through constrained geometries in Stokes flow. In this work, we formulate an adjoint-based optimization approach that overcomes several of the associated computational bottlenecks.\enlargethispage*{2ex}

In \cite{bonnet2020shape}, we considered the shape optimization of peristaltic pumps transporting a simple Newtonian fluid at low Reynolds numbers, which in turn was inspired by the work of Walker and Shelley \cite{walk:shel:10}. In contrast, the present work considers the case of pumps transporting large solid particles suspended in the viscous fluid (schematic in Figure \ref{geom:2D}). This extension, however, is non-trivial, since a dynamic fluid-structure interaction problem needs to be solved to simulate the transport for a given peristaltic wave shape.  

The main contributions of this work are two-fold. First, to evaluate the shape sensitivities efficiently, we systematically derive adjoint formulations for all the required shape functionals.  
The new shape derivative formulas require evaluating physical and adjoint variables on the domain boundaries only, consistently with the general structure of shape derivative formulas~\cite{henrot:pierre:18}. Adjoint formulations are very widely used for the evaluation of shape or material sensitivities in PDE-constrained optimization~\cite{hinze:08}, even in situations involving time-dependent forward problems, as here. They have recently been applied to droplet shape control in~\cite{fikl:21} and are also commonplace in applications such as geophysical full waveform inverse 
problems~\cite{epanomeritakis:08,metivier:12}. Our proposed shape sensitivity formulas allow, for each shape functional involved in the present optimization problem, to evaluate its derivatives with respect to any chosen set of shape parameters by using a single time-backwards adjoint solution. While this characteristic is relatively classical nowadays, we faced and solved a significant and less-common additional difficulty, namely that the fluid carries particles whose motion depends on the shape being optimized in an \emph{a priori} unknown way and gives rise to design-dependent time-evolving shape parameters. Our adjoint problems are designed so that the contribution of the latter is accounted for, circumventing the need of evaluating explicitly the shape sensitivity of the motion of carried particles.

Second, as in \cite{bonnet2020shape}, we employ boundary integral equation (BIE) techniques to solve the governing equations. Their usual advantages over classical domain discretization methods---reduction in dimensionality, high-order accuracy and availability of fast solvers---are particularly significant for the shape optimization considered here as they avoid the need for volume re-meshing between optimization iterations and across the time steps. Specifically, we adapt the BIE method developed in \cite{marple:16} to solve our forward and the associated adjoint problems. In contrast to classical BIE techniques that employ periodic Green’s functions, it uses free-space Green's functions together with a set of auxiliary sources and enforces the periodic boundary conditions algebraically. The forward and adjoint problems require enforcing a variety of boundary conditions on the channel and particle boundaries as well as jump conditions across the channel---all of which can be accomodated in a straightforward manner using this BIE formulation.

This paper is organized as follows. In Section \ref{sc:formulation}, we introduce the PDE formulation of the peristaltic pumping problem and formally define the shape optimization problem.  
The shape sensitivities of the objective function and the constraints are derived in Section \ref{sc:sensitivities}.
The boundary integral method for solving the forward and adjoint problems and the numerical optimization procedure are discussed in Section \ref{sc:numerical}. 
Validation tests and the optimal shapes under various constraints are presented in Section \ref{sc:results}, followed by conclusions in Section \ref{sc:conclusions}.

\section{Problem formulation} \label{sc:formulation}

\subsection{Formulation of the wall motion}
\label{geometry}

Pumping is achieved by the channel wall shape moving along the positive direction $\bfe_1$ at a constant velocity $c$, as a traveling wave of wavelength $L$ (the wave period therefore being $T\per\shdeq L/c$). The quantities $L,c$ are considered as fixed in the wall shape optimization process. This apparent shape motion is achieved by a suitable material motion of the wall, whose material is assumed to be flexible but inextensible. Like in~\cite{walk:shel:10}, it is convenient to introduce a \emph{wave frame} that moves along with the traveling wave, i.e., with velocity $c\bfe_1$ relative to the (fixed) lab frame.

\begin{figure}[t] \centering
  \includegraphics[width=0.9\textwidth]{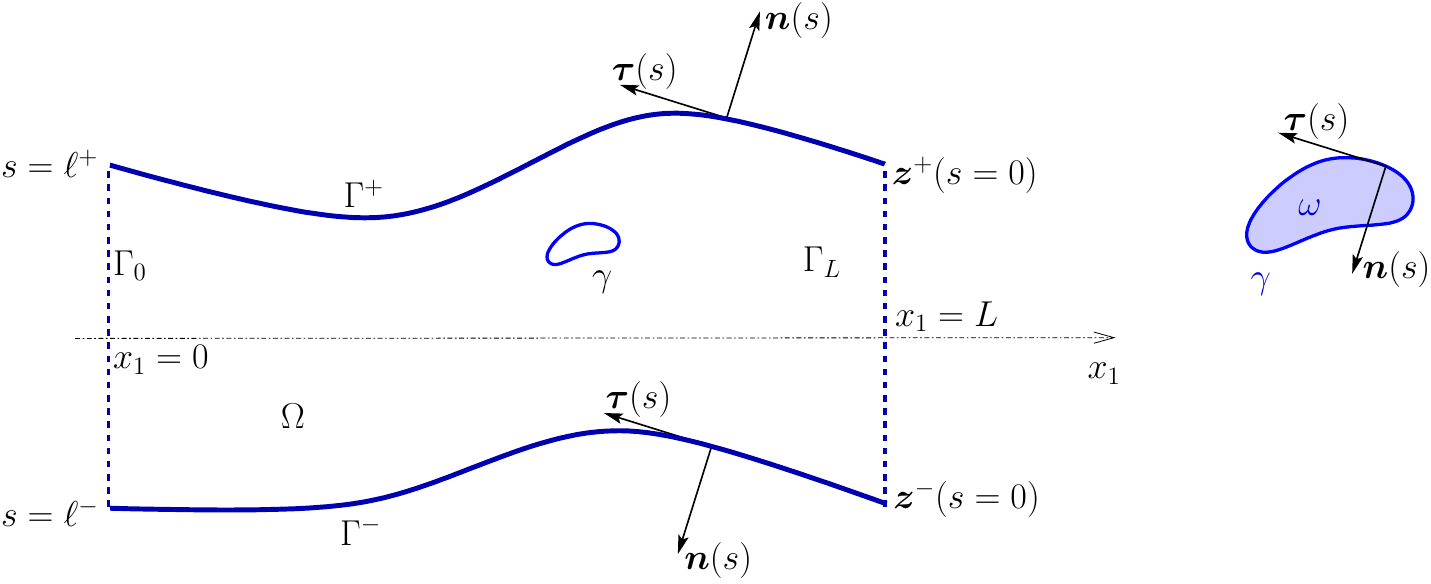}
  \caption{2D periodic channel with particle-carrying flow in wave frame: geometry and notation}\label{geom:2D}
\end{figure}

Here we consider fluid flows carrying rigid particles, treating in detail the case of one such particle. The particle motion makes the flow, and the fluid domain, time-dependent, and we denote by $t\shin[0,T]$ the time interval of interest, the duration $T$ being arbitrary. Let $\Omega(t)$ denote, in the wave frame, the fluid region enclosed in one wavelength of the channel (see Fig.~\ref{geom:2D}), and let $\oo(t)$ and $\gamma(t)$ be the domain occupied at time $t$ in the wave frame by the particle and its closed contour. The fluid domain boundary is $\dO(t)=\G\shcup\G_p\shcup\gamma(t)$. The wall $\G\shdeq\G^{+}\shcup\G^{-}$, which is fixed in this frame, has disconnected components $\G^{\pm}$ which are not required to achieve symmetry with respect to the $x_1$ axis and have respective lengths $\ell^{\pm}$. The remaining channel contour $\G_p\shdeq\G_0\shcup\G_L$ consists of the periodic planar end-sections $\G_0$ and $\G_L$, respectively situated at $x_1=0$ and $x_1=L$; the endpoints of $\G_L$ are denoted by $\bfz^{\pm}$ (Fig.~\ref{geom:2D}). The orientation conventions of Fig.~\ref{geom:2D} are used throughout.

The fluid flow at any given time is assumed to be spatially periodic in the channel axis direction. This implies a periodic arrangement of the carried particle(s); for instance, the single particle considered in what follows is implicitly replicated in each periodic segment of the channel.

\subsubsection*{Wall geometry and motion} Both channel walls are described as arcs $s\mapsto \bfx^{\pm}(s)$ with the arclength coordinate $s$ directed ``leftwards'' as depicted in~Fig.~\ref{geom:2D}, whereas the unit normal $\bfn$ to $\dO$ is everywhere taken as outwards to $\OO$. The position vector $\bfx(s)$, unit tangent $\bftau(s)$, unit normal $\bfn(s)$ and curvature $\kappa(s)$ obey the Frenet formulas
\begin{equation}
  \del{s}\bfx = \bftau, \qquad \del{s}\bftau = \kappa\bfn, \qquad
  \del{s}\bfn = -\kappa\bftau \qquad \text{on $\G^+$ and $\G^-$}. \label{frenet}
\end{equation}
For consistency with our choice of orientation convention (and with the above formulas), the curvature is everywhere on $\G$ taken as $\kappa\sheq\bfn\sip\del{s}\bftau$.\enlargethispage*{1ex}

In the wall frame, the wall particle velocity must be tangent to $\G$ (wall material points being constrained to remain on the surface $\G$); moreover the wall material is assumed to be inextensible. In the wave frame, the wall particle velocities $\bfU$ satisfying both requirements must have, on each wall, the form
\begin{equation}
  \bfU(s) = U\bftau(s),
\end{equation}
where $U$ is constant. Moreover, in the wave frame, all wall material points travel over an entire spatial period during the time interval $T\per=1$, which implies $U=\ell$ ($\ell$ being the ratio between wall and channel lengths due to scaling). Finally, the viscous fluid must obey a no-slip condition on the wall, so that the velocity of fluid particles adjacent to $\bfx(s)$ is $\bfU(s)$. Concluding, the pumping motion of the wall constrains on each wall the fluid motion through
\begin{equation}
  \bfu(\bfx) = \bfuD(\bfx) := \ell^{\pm}\bftau^{\pm}(\bfx) ,\quad \bfx\in\G^{\pm}. \label{noslip:2D}
\end{equation}
In the sequel, we will drop plus, minus and $\pm$ symbols referring to upper and lower channel walls, with the understanding that quantities attached to $\G$ (e.g. $\ell$) may take distinct values on either wall.

\subsubsection*{Rigid particle motion} The motion of material points $\bfx$ of a rigid particle $\oo$, or its contour $\gamma$, has the Lagrangian representation
\begin{equation}
  \oo(t) = \lcb \bfxr(\bfx_0,t), \quad \bfx_0\shin\oo_0 \rcb, \qquad \text{with}\quad
  \bfxr(\bfx_0,t) := \bfc(t) + \bfR(t)\sip\bfx_0
\label{rigid:motion}
\end{equation}
where $\bfx_0$ is the position of the material point at initial time $t\sheq0$ and $\oo_0\shdeq\oo(0)$ the initial configuration of the particle, while the time-dependent vector $\bfc(t)$ (with $\bfc(0)\sheq\bfze$) and the time-dependent unitary matrix $\bfR(t)\shin\text{SO}(2)$ (with $\bfR(0)\sheq\bfI$) respectively describe the particle translation and rotation relative to the initial particle configuration $\oo_0$. The corresponding particle velocity $\bdot{\bfx}$ is, in Eulerian form:
\begin{equation}
  \bdot{\bfx}(t)
 = \bfw(t) + \varrho(t)\,\bfe_3\shtimes\bfx = \bfw(t) + \varrho(t)\,\bfr\sip\bfx, \qquad \bfx\shin\oo(t), \label{rigid:velocity}
\end{equation}
with the constant skew-symmetric tensor $\bfr$ defined by $\bfr\shdeq\bfe_2\tensor\bfe_1-\bfe_1\tensor\bfe_2$ and the angular velocity $\varrho$ and translational velocity $\bfw$ linked at any time $t$ to $\bfR,\bfc$ through $\varrho \bfr = \bdot{\bfR}\sip\bfR\Tsup$ and $\bfw = \bdot{\bfc}-\varrho\bfe_3\shtimes\bfc$.

\subsection{Forward problem: PDE formulation}\label{forward:pde}

The fluid is assumed to be viscous (with dynamic viscosity $\mu$) and incompressible, so that the stress tensor is given by
\begin{equation}
  \bfsig[\bfu,p] = -p\bfI + 2\mu\bfD[\bfu] \label{constitutive}
\end{equation}
where $\bfD[\bfu]:=\tdemi(\bfna\bfu\shp\bfna\Tsup\bfu)$ is the strain rate tensor and $p$ is the pressure. We henceforth use the parameters $L,c,\mu$ to define non-dimensional of all relevant variables: coordinates and lengths are scaled by $L$, velocities by $c$, angular velocities by $c/L$, time by $L/c$, and stresses (including traction vectors and pressures) by $\mu c/L$. All geometrical or physical variables appearing thereafter are implicitly non-dimensional, after scaling according to the foregoing conventions.

The particle-carrying flow in the wave frame~\cite{walk:shel:10} during a time interval $t\shin[0,T]$ is described at any time instant by the incompressible Stokes equations with periodicity conditions
\begin{subequations}
\begin{equation}
  -\Delta\bfu + \bfna p = \bfze, \quad \Div\bfu= 0 \quad \text{in $\OO(t)$}, \qquad \bfu|_{\G_L} = \bfu|_{\G_0}. \label{forward:PDE}
\end{equation}
The fluid motion results from the prescribed wall velocity
\begin{equation}
  \bfu = \bfuD := \ell\bftau \qquad \text{on $\G$}. \label{forward:PDE:wall}
\end{equation}
The rigid particle in turn undergoes a rigid-body motion of the form~\eqref{rigid:motion} due to being carried by the fluid through the no-slip condition at any time:
\begin{equation}
  \bfu = \bdot{\bfx}, \quad  \qquad \text{on $\gamma(t)$} \label{forward:PDE:particle}
\end{equation}
The fluid region $\OO(t)$ and the flow solution are time-dependent due to the particle motion $t\mapsto \oo(t)$. Equations~(\ref{forward:PDE}-c) define a well-posed problem for $(\bfu,p)$ at any time $t$ if the particle motion (and hence $\bdot{\bfx}$ on $\gamma(t)$) is known, $p$ being determined up to an arbitrary (and irrelevant) additive constant.\enlargethispage*{7ex}

The particle motion being, in fact, unknown, it is determined from the condition that the hydrodynamic forces exerted on $\gamma(t)$ have a zero net force and net torque, i.e.:
\begin{equation}
  \igt \bfsig[\bfu,p]\sip\bfn \ds = \bfze, \qquad
  \igt \lpar\bfsig[\bfu,p]\sip\bfn\rpar\sip(\bfe_3\shtimes\bfx) \ds = 0 \qquad t\in[0,T]. \label{equilibrium}
\end{equation}
Conditions~\eqref{equilibrium} allow to determine the three DOFs $\bfw(t)$ and $\varrho(t)$ of the particle velocity $\bdot{\bfx}$, see~\eqref{rigid:velocity}. The particle motion is then found by integrating $\bdot{\bfx}$ in time from a given initial condition
\begin{equation}
  \oo(0) = \oo_0. \label{gamma0}
\end{equation}
\end{subequations}
The solution of the forward evolution problem~\eqref{forward:PDE}--\eqref{gamma0}, and in particular the particle motion, is entirely determined by the shape of the wall $\G$, since the data $\bfuD$ given by~\eqref{noslip:2D} is. In a time-discrete explicit setting with time step $\Delta t\sheq T/N$ and time instants $t_0\sheq0, t_1\sheq\Delta t,\ldots,t_N\sheq T$, equations~(\ref{forward:PDE}-d) are solved at each $t\sheq t_n$ and the particle configuration $\gamma_n\shdeq\gamma(t_n)$ is updated in explicit fashion through
\begin{equation}
  \gamma_{n+1} = \gamma_n + \bdot{\bfx}(\gamma_n,t_n)\Delta t \label{gamma:update}
\end{equation}

\subsection{Forward problem: weak formulation}\label{sec:forward:weak}

In this work, flow computations rely on a boundary integral equation (BIE) formulation of equations~(\ref{forward:PDE}-d), see Sec.~\ref{sec:forwardsolver}. It is however convenient, for the derivation of shape derivative identities and adjoint problems, to recast equations~\eqref{forward:PDE}--\eqref{equilibrium} of the forward evolution problem in the following mixed weak form (e.g. \cite{brez:fort:91}, Chap. 6):
\begin{equation} \label{forward:weak}
\begin{multlined}
\text{For each $t\shin[0,T]$, find }(\bfu,p,\bff,\bfh,\bdot{\bfx}) \in \US\shtimes\Pcal\shtimes\FS\shtimes\HS\shtimes\RS, \\
\left\{
\begin{aligned}
  \text{(a) }&& a(\bfu,\bfv) - b(\bfu,q) - b(\bfv,p) - \lbra\bff,\bfv\rbra_{\G} - \lbra\bfh,\bfv\rbra_{\!\gamma(t)} &= 0
  &\quad& \forall(\bfv,q)\shin\US\shtimes\Pcal\hspace*{-1em} \\
  \text{(b) }&& \lbra\bfg,\bfuD\rbra_{\G} - \lbra\bfg,\bfu\rbra_{\G} &= 0 && \forall\bfg\shin\FS \\
  \text{(c) }&& \lbra\bfk,\bdot{\bfx}\rbra_{\!\gamma(t)} - \lbra\bfk,\bfu\rbra_{\!\gamma(t)} &= 0 && \forall\bfk\shin\HS \\
  \text{(d) }&& \lbra \bfh, \bfrh \rbra_{\!\gamma(t)} &= 0 && \forall\bfrh\shin\RS
\end{aligned} \right.\hspace*{-0.5em}
\end{multlined}
\end{equation}
where $\lbra\cdot,\cdot\rbra_{X}$ stands for the $L^2(X)$ duality product, and the bilinear forms $a$ and $b$ are defined by
\begin{equation}
  a(\bfu,\bfv) = \iO 2\bfD[\bfu]\dip\bfD[\bfv] \dV, \qquad b(\bfv,q) = \iO q\,\Div\bfv \dV
\end{equation}
The function spaces in equations~\eqref{forward:weak} are as follows: $\US$ is the space of all periodic vector fields contained in $H^1(\OO;\Rbb^2)$, $\Pcal$ is the space of all $L^2(\OO)$ functions with zero mean (i.e. obeying the constraint $\lbra p,1 \rbra_{\OO}\sheq0$), $\FS\sheq H^{-1/2}(\G;\Rbb^2)$ and $\HS\sheq H^{-1/2}(\gamma;\Rbb^2)$. The dependence on time of $\US,\Pcal,\FS$ (through the time-dependent regions $\OO(t)$ and $\oo(t)$) is implicitly understood. The chosen definition of $\Pcal$ caters for the fact that $p$ would otherwise be defined only up to an arbitrary additive constant. The Dirichlet boundary conditions~\eqref{forward:PDE:wall} and~\eqref{forward:PDE:particle} are (weakly) enforced through~(\ref{forward:weak}b,c), rather than being embedded in the velocity space $\US$, as this will make the derivation of shape derivative identities simpler. The unknown $\bff$, which acts as the Lagrange multiplier associated with condition~\eqref{forward:PDE:wall}, is in fact the force density (i.e stress vector) $\bfsig[\bfu,p]\sip\bfn$ on $\G$; likewise, $\bfh\shdeq\bfsig[\bfu,p]\sip\bfn$ is the stress vector arising on $\gamma$ from the kinematic condition~\eqref{forward:PDE:particle}. Condition~\eqref{equilibrium} is then the weak form of condition~\eqref{equilibrium}, $\RS$ being the three-dimensional space of rigid-body velocity fields
\begin{equation}
  \RS:=\lcb \bfrh=\bfrhh+\hatr\bfe_3\shtimes\bfx\,,\; (\bfrhh,\hatr)\shin\Rbb^2\shtimes\Rbb \rcb. \label{rho:def}
\end{equation}
Equations~\eqref{forward:weak} govern the flow at each instant $t$, knowing the current particle position $\oo(t)$. The complete forward evolution problem in weak form consists of~\eqref{forward:weak} supplemented with the initial condition~\eqref{gamma0}, with the particle motion $\oo(t)$ again to be found by integrating $\bdot{\bfx}$ in time.

\subsection{Objective functionals and optimization problem}\label{optim:pb}

We seek channel wall shapes that optimize the efficiency of peristaltic pumping. This problem involves three main quantities, namely the dissipation, the net particle motion and the mass flow rate, which we first describe.

\subsubsection*{Dissipation} The dissipation over a chosen duration $T$ is given~\cite{walk:shel:10} by the functional
\begin{equation}\label{J:loss:def}
  J\Wsub(\G)
 := \iT \Lcb \lbra \bff,(\bfuD\shp \bfe_1) \rbra_{\G} + \lbra \bfh,(\bfu\shp \bfe_1) \rbra_{\!\gamma(t)} \Rcb \dt
 = \iT \lbra \bff,(\bfuD\shp \bfe_1) \rbra_{\G} \dt
\end{equation}
(up to the scaling factor $\mu cL$) where $(\bff,\bfh,\bfu)$ are components of the forward solution at time $t$ and the last equality stems from~\eqref{equilibrium}. Its value being completely determined by the shape of the wall $\G$ (in a partly implicit way through $\bff$ and the $\G$-dependent particle evolution $\gamma(t)$), $J\Wsub$ is a shape functional. \enlargethispage*{1ex}

\subsubsection*{Net particle motion} The net motion $D(\G)\shdeq \rmx_1\Gsup(T)\shm\rmx_1\Gsup(0)$ along $\bfe_1$ of the particle centroid $\bfxr\Gsup(t)$ in the wave frame is given by
\begin{equation}
  |\oo|D(\G) = \lbra x_1,1 \rbra_{\!\omega(T)} - \lbra x_1,1 \rbra_{\!\omega_0} \label{xG:def}
\end{equation}

\subsubsection*{Optimization problem}

Consider a given particle initial domain $\omega_0$ and a chosen duration $T$, the goal is to find the optimal wall shape $\G$ of the peristaltic pumping channel that minimizes the dissipation functional $J\Wsub$ subject to the volume $|\OO|$ of the fluid region being constant and the net particle motion $D$ (in the wave frame) being given. In the fixed frame, the net particle motion is $D\shp T$ and the corresponding net particle velocity is $D/T\shp1$. The constrained optimization problem is then:
\begin{equation}
 \G^{\star} =  \argmin_{\OO(\G) \, \shin\, \Ocal} \; J\Wsub(\G) \quad
 \quad\text{subject to } \quad  \left\{\begin{aligned} C\Vsub(\G) &:= |\OO(\G)| \shm V_0 &=0, \\
 C\Dsub(\G) &:= D(\G) \shm D_0 &=0, \end{aligned}\right.
\label{eq:opt}\end{equation}
where $\Ocal$ is the set of admissible shapes of $\OO$ (see Sec.~\ref{overview}) and $V_0,D_0$ are given target values.

\subsubsection*{Mass flow rate} Another quantity frequently involved in the optimization of flows in channels is the average mass flow rate per wavelength $Q(\G)$, defined in the wave frame by
\begin{equation}
  Q(\G) = \inv{T}\int_0^T \iOt (\bfu\shp \bfe_1)\sip\bfe_1 \dV\dt
  = |\OO| + \inv{T}\int_0^T \lbra u_1,1 \rbra_{\OO(t)} \dt, \label{flowrate}
\end{equation}
(up to the scaling factor $cL$ and with $u_1\sheq \bfu\sip\bfe_1$).
We next observe that $\bfu$ is a rigid-body velocity on $\gamma(t)$, and can thus be continuously extended inside $\oo(t)$ as the particle rigid-body velocity field $\bdot{\bfx}$, so that
\begin{equation}
  \lbra u_1,1 \rbra_{\OO(t)} = \lbra u_1,1 \rbra_{\OO(t)\cup\oo(t)} - \lbra \bdot{x}_1,1 \rbra_{\oo(t)}
 = \lbra u_1,1 \rbra_{\OO(t)\cup\oo(t)} - |\oo|\bdot{x}{}_1\Gsup(t)
\end{equation}
(with $u_1$ in the integral over $\OO(t)\cup\oo(t)$ understood as the above-introduced extension). Since $\bfu$ is divergence-free in $\OO(t)\shcup\oo(t)$, we have $u_1=\Div(x_1\bfu) - x_1\Div\bfu=\Div(x_1\bfu)$ and the divergence theorem provides
\begin{equation}
  \lbra u_1,1 \rbra_{\OO(t)\cup\oo(t)} = \lbra x_1\bfu,\bfe_1 \rbra_{\G_L} -  \lbra x_1\bfu,\bfe_1 \rbra_{\G_0}
 = \lbra u_1,1 \rbra_{\G_L}.
\end{equation}
The average mass flow rate per wavelength is finally given by
\begin{equation}
    Q(\G)
 = |\OO| \shp C(\G) \shm \frac{|\oo|}{T}D(\G), \qquad C(\G) = \inv{T}\int_0^T \lbra u_1,1 \rbra_{\G_L} \dt 
\label{flowrate:expr},
\end{equation}
as a combination of boundary and particle integrals, a format that is well suited to the present use of BIE solvers. We note that thanks to the above-discussed velocity field extension, the last integral in~\eqref{flowrate:expr} involves the whole end section $\G_L$ irrespective of whether the particle crosses it at some particular time.\enlargethispage*{3ex}

Even though the mass flow rate is not involved in the examples presented in Section~\ref{sc:results}, we will derive and provide its shape derivative as a useful additional result, see Section~\ref{dd:explicit}.

\section{Shape sensitivities} \label{sc:sensitivities}

This section begins with an overview of available shape derivative concepts that also serves to set notation (Sec.~\ref{overview}). We then derive the governing problem for the shape derivative of the forward solution (Sec.~\ref{dd:forward}) and use this result to formulate shape derivatives of objective functionals in terms of an adjoint solution (Sec.~\ref{dd:J}). Specific cases of functionals are finally addressed in Sec.~\ref{dd:explicit}.

\subsection{Shape sensitivity analysis: an overview}
\label{overview}

We begin by collecting available shape derivative concepts that fit our needs, referring to e.g.~\cite[Chaps. 8,9]{del:zol} or \cite[Chap. 5]{henrot:pierre:18} for rigorous expositions of shape sensitivity theory. Let $\OO_{\text{all}}\shsubs\Rbb^2$ denote a fixed domain chosen so that $\OO\Subset\OO_{\text{all}}$ always holds for the shape optimization problem of interest. Upon introducing \emph{transformation velocity} fields, i.e. vector fields $\bfth:\OO_{\text{all}}\to\Rbb^2$ such that $\bfth=\bfze$ in a neighborhood of $\dO_{\text{all}}$, shape perturbations of domains $\OO\Subset\OO_{\text{all}}$ are mathematically described using a pseudo-time $\eta$ and a geometrical transformation of the form
\begin{equation}
  \bfx\shin\OO_{\text{all}} \mapsto \bfx^\eta = \bfx+\eta\bfth(\bfx), \label{shape:pert}
\end{equation}
which defines a parametrized family of domains $\OO_{\eta}(\bfth):=(\bfI\shp\eta\bfth)(\OO)$ for any given ``initial'' domain $\OO\Subset\OO_{\text{all}}$. The affine format~\eqref{shape:pert} is sufficient for defining the first-order derivatives at $\eta\sheq0$ used in this work.

\subsubsection*{Admissible shapes and their transformations} The set $\Ocal$ of admissible shapes for the fluid region $\OO$ in a channel period (in the wave frame) is chosen as
\begin{equation}
  \Ocal = \big\{ \OO\Subset\OO_{\text{all}}, \text{ $\OO$ is periodic and connected} \big\}.
  \label{Ocal:def}
\end{equation}
Accordingly, let the space $\Theta$ of admissible transformation velocities be defined as
\begin{equation}\label{Theta:def}
  \Theta = \Lcb \bfth\shin C_0^{1,\infty}(\OO_{\text{all}})\ \big|\ \text{(i) }
  \bfth|_{\Gpp}\sheq\bfth|_{\Gpm}, \ \text{(ii) }\bfth\sip\bfe_1\sheq0 \text{ on $\Gpp$},\ \text{(iii) }\bfth(\bfz^-)\sheq\bfze \Rcb,
\end{equation}
(where $C_0^{1,\infty}(\OO_{\text{all}}):=W^{1,\infty}(\OO_{\text{all}})\cap C_0^{1}(\OO_{\text{all}})$) ensuring that the shape perturbations (i) are periodic, (ii) prevent any deformation of the end sections $\Gp^{\pm}$ along the axial direction, and (iii) prevent vertical rigid translations of the channel domain. The provision $\bfth\in C_0^{1,\infty}(\OO_{\text{all}})$ ensures that (a) there exists $\eta_0\shg0$ such that $\OO_{\eta}(\bfth)\Subset\OO_{\text{all}}$ for any $\eta\shin[0,\eta_0]$, (b) the weak formulation for the shape derivative of the forward solution (see~\eqref{der:weak}) is well defined in the standard solution spaces, and (c) traces of $\bfth$ and $\bfna\bfth$ on $\dO_{\eta}$ are well-defined. Since here shape changes are driven by $\G$, the support of $\bfth$ may be limited to an arbitrary neighborhood of $\G$ in $\OO$.\enlargethispage*{3ex} 

\subsubsection*{Lagrangian derivatives}

In what follows, all shape derivatives are implicitly taken at some given configuration $\OO$, i.e. at initial "time" $\eta\sheq0$. The ``initial'' Lagrangian derivative $\dd{\bfa}$ of some (scalar or tensor-valued) field variable $\bfa(\bfx,\eta)$ is defined as
\begin{equation}
  \dd{\bfa}(\bfx)
 = \lim_{\eta\to0} \inv{\eta}\lsqb \bfa(\bfx^{\eta},\eta)-\bfa(\bfx,0) \rsqb \qquad\bfx\shin\OO, \label{dd:def}
\end{equation}
and the Lagrangian derivative of gradients and divergences of tensor fields are given by
\begin{equation}
  \text{(a) \ }(\bfna\bfa)^{\star} = \bfna\dd{\bfa} - \bfna\bfa\sip\bfna\bfth, \qquad
  \text{(b) \ }(\Div\bfa)^{\star} = \Div\dd{\bfa} - \bfna\bfa\dip(\bfna\bfth)\Tsup \label{dd:grad}
\end{equation}
Likewise, the first-order ``initial'' directional derivative $J'$ of a shape functional $J:\Ocal\to\Rbb$ is defined as
\begin{equation}
  J'(\OO;\bfth) = \lim_{\eta\to0} \inv{\eta}\lpar J(\OO_{\eta}(\bfth))-J(\OO) \rpar. \label{dd:J:def}
\end{equation}
In this work, Lagrangian derivatives with respect to the pseudo-time $\eta$ and the physical time $t$ are distinguished by being respectively called  ``Lagrangian'' and ``particle'' derivatives, and denoted using a star (e.g. $\dd{\bfa}$) or a dot (e.g. $\bdot{\bfx}$).

\subsubsection*{Lagrangian differentiation of integrals}

Consider, for a given transformation velocity field $\bfth\shin\Theta$, generic domain and contour integrals
\begin{equation}
  \text{(a) \ }I\Vsub(\eta) = \int_{\OO_{\eta}(\bfth)} F(\cdot,\eta) \dV, \qquad
  \text{(b) \ }I\Ssub(\eta) = \int_{S_{\eta}(\bfth)} F(\cdot,\eta) \ds, \label{IVS:def}
\end{equation}
where $\OO_{\eta}(\bfth)=(\bfI\shp\eta\bfth)(\OO)$ is a variable domain and $S_{\eta}(\bfth)\shdeq(\bfI\shp\eta\bfth)(S)$ a (possibly open) variable curve. The derivatives of $I\Vsub(\eta)$ and $I\Ssub(\eta)$ are given by
\begin{equation}
\begin{aligned}
  &\text{(a) \ }& \der{I\Vsub}{\eta}\Big|_{\eta=0} &= \iO \lsqb \dd{F} + F(\cdot,0)\,\Div\bfth \rsqb \dV, \\
  &\text{(b) \ }& \der{I\Ssub}{\eta}\Big|_{\eta=0} &= \iS \lsqb \dd{F} + F(\cdot,0)\,\DivS\bfth \rsqb \ds,
\end{aligned} \label{dd:IVS}
\end{equation}
which are well-known material differentiation formulas of continuum kinematics. In (\ref{dd:IVS}b), $\DivS$ is the tangential divergence operator, given in the present 2D context by
\begin{equation}
  \DivS\bfth = \lpar\bfI\shm\bfn\tensor\bfn\rpar\dip\bfna\bfth = \del{s}\theta_s-\kappa\theta_n
\label{DivS:def}
\end{equation}
where $\bfth$ on $\G$ is set, using the unit vectors defined in~\eqref{frenet}, in the form $\bfth=\theta_s\bftau+\theta_n\bfn$ and the curvature $\kappa$ also follows the conventions of~\eqref{frenet}.\enlargethispage*{3ex}

Finally, the following simple result (proved in Appendix~\ref{DivS:theta:proof}) will prove useful, as we will consider particle motions, and geometrical transformations of particle-carrying fluid regions, that preserve the particle shape:
\begin{lemma}\label{DivS:theta}
Let $\bfw\shin\RS$ be a rigid-body vector field on $\omega$, and let $\bfu\shin\US$ denote any extension of $\bfw$ in $\OO$ satisfying $\bfu|_{\gamma}=\bfw$. Then: $\DivS\bfu = 0$ on $\gamma$.
\end{lemma}

\subsubsection*{Shape functionals and structure theorem}

The structure theorem for shape derivatives (see e.g.~\cite[Chap. 8, Sec. 3.3]{del:zol}) then states that the derivative of any shape functional $J$ is a linear functional in the normal transformation velocity $\theta_n\sheq\bfn\sip\bfth|_{\dO}$. For PDE-constrained shape optimization problems involving sufficiently smooth domains and data, the derivative $J'(\OO;\bfth)$ has the general form
\begin{equation}
  J'(\OO;\bfth) = \idO g\, \theta_n \ds, \label{shape:structure}
\end{equation}
where $g$ is the \emph{shape gradient} of $J$: intuitively speaking, the shape of $\dO_{\eta}$ determines that of $\OO_{\eta}$ while the tangential part of $\bfth$ leaves $\OO_{\eta}$ unchanged at leading order $O(\eta)$.

\subsubsection*{Example: derivative of channel volume}
The channel volume $V(\OO)\shdeq|\Omega|$ being given by~(\ref{IVS:def}a) with $F=1$, identity~(\ref{dd:IVS}a) and Green's theorem readily yield
\begin{equation}
  V'(\OO;\bfth) = \iO \Div\bfth \dV = \idO \theta_n \ds = \iG \theta_n \ds, \label{dd:volume}
\end{equation}
the last equality being due to provision (ii) in~\eqref{Theta:def} and the rigid particle motion.\enlargethispage*{5ex}

\subsection{Shape derivative of the forward solution}
\label{dd:forward}

The functionals introduced in Sec.~\ref{optim:pb} depend on $\G$ implicitly through the forward solution $(\bfu,p,\bff,\bfh,\bfxr)$. Finding their shape derivatives then involves the forward solution derivative  $(\dd{\bfu},\dd{p},\dd{\bff},\dd{\bfh},\dd{\bfx})$. Unlike in the earlier study~\cite{bonnet2020shape}, here the flow domain evolves in time in a manner that is not \emph{a priori} known. Towards setting up the governing problem for  $(\dd{\bfu},\dd{p},\dd{\bff},\dd{\bfh},\dd{\bfx})$, we thus begin by formulating the sensitivity of particle evolution to the shape of the channel wall.

Perturbations of the wall shape, described through geometrical transformations of the form~\eqref{shape:pert}, induce perturbations of the particle motion through the evolution problem~(\ref{forward:PDE}-\ref{gamma0}), which can be described by making the rigid-body motion~\eqref{rigid:motion} dependent on $\eta$. Hence, for any material point of $\oo_{\eta}(t)$, we have
\begin{equation}
  \bfxr^{\eta} = \bfxr(\bfx_0,t,\eta) := \bfc(t,\eta) + \bfR(t,\eta)\sip\bfx_0, \quad\bfx_0\shin\oo_0
\label{rigid:motion:eta}
\end{equation}
The Lagrangian derivative at $\eta\sheq0$ of a point $\bfx^{\eta}$ of $\oo_{\eta}(t)$ following the shape transformation, being defined by $\bfx^{\eta} = \bfx \shp \eta\dd{\bfx} + o(\eta)$ whenever such expansion exists, is thus given by
\begin{equation}
  \dd{\bfx} = \dd{\bfxr}(\bfx_0,t) = \del{\eta}\bfc(t,0) + \del{\eta}\bfR(t,0)\bfx_0 \quad\text{at \ } \bfx=\bfxr(\bfx_0,t)\in\oo(T)
\end{equation}
provided $\bfx,\bfc$ depend smoothly enough on $\eta$, and is moreover readily found to be a rigid-body velocity (since $\bfx_0\sheq\lsqb\bfR\Tsup\sip(\bfx\shm\bfc)\rsqb(t,0)$ and $\lsqb\del{\eta}\bfR\sip\bfR\Tsup\shp\bfR\sip\del{\eta}\bfR\Tsup\rsqb(t,0)\sheq\bfze$). In addition, the particle motion being assumed for each $\eta$ to start from the same initial particle $\oo_0$, we have $\bfc(0,\eta)\sheq\bfze,\;\bfR(0,\eta)\sheq\bfI$ and hence
\begin{equation}
  \dd{\bfxr}(\cdot,0) = \bfze. \label{beta:init}
\end{equation}
Finally, as the no-slip condition~\eqref{forward:PDE:particle} remains true for any small enough $\eta$ (i.e. $\bfu^{\eta}=\bdot{\bfx}{}^{\eta}$), we find that
\begin{equation}
  \dd{\bfu}(\bfx,t) = (\dd{\bfx}^{\Bdot}\!\!\!)(\bfx,t), \qquad \bfx\shin\gamma(t) \label{dd:u:beta}
\end{equation}

Since (again) we have $\bfx^{\eta} = \bfx \shp \eta\dd{\bfx} + o(\eta)$ in $\oo(t)$, $\dd{\bfx}$ is the transformation velocity for perturbations $\oo^{\eta}(t)$ of the particle $\oo(t)$. Sensitivies of integrals over $\OO$ or $\gamma$ with respect to the shape of $\gamma$ are therefore given by~\eqref{dd:IVS} with $\bfth,S$ replaced by $\dd{\bfx},\gamma$. The support of any (arbitrary) extension of $\dd{\bfx}$ to $\OO$ required in~(\ref{dd:IVS}a) may be limited to a neighborhood of $\gamma$ in $\OO$. In fact, if the particle motion avoids any contact with the channel wall, we may assume that $\text{supp}(\bfth)\shcap\oo\sheq\emptyset$ and $\text{supp}(\dd{\bfx}\shcap\G)\sheq\emptyset$.

We are now ready to formulate the shape derivative of the forward solution:\enlargethispage*{3ex}
\begin{proposition}\label{lm1}
The shape derivative $(\dd{\bfu},\dd{p},\dd{\bff},\dd{\bfh},\dd{\bfx})$ of the forward solution $(\bfu,p,\bff,\bfh,\bfxr)$ satisfies
\begin{equation}\begin{multlined}
\text{For each $t\shin[0,T]$, find }(\dd{\bfu},\dd{p},\dd{\bff},\dd{\bfh},\dd{\bfx}) \in \US\shtimes\Pcal\shtimes\FS\shtimes\HS\shtimes\RS, \\ \left\{
\begin{aligned}
\text{(a) \ }& a(\dd{\bfu},\bfv) -b(\dd{\bfu},q) - b(\bfv,\dd{p}) - \lbra\dd{\bff},\bfv\rbra_{\G}
 - \lbra\dd{\bfh},\bfv\rbra_{\!\gamma(t)} + \lbra\bfE\lpar(\bfu,p),(\bfv,q)\rpar,\bfna\Tsup\dd{\bfx} \rbra_{\OO(t)}
 \suite\qquad\qquad
 = -\lbra\bfE\lpar(\bfu,p),(\bfv,q)\rpar,\bfna\Tsup\bfth \rbra_{\OO(t)} + \lbra\bff,\bfv\DivS\bfth\rbra_{\G} & \forall(\bfv,q)&\shin\US\shtimes\Pcal,\hspace*{-1em} \\
\text{(b) \ }&	\lbra\dd{\bfu}{}\Dsup,\bfg\rbra_{\G} - \lbra\dd{\bfu},\bfg\rbra_{\G} = 0 & \forall\bfg &\shin\FS, \\
\text{(c) \ }&	\lbra (\dd{\bfx}^{\Bdot}\!\!\!),\bfk\rbra_{\!\gamma(t)} - \lbra\dd{\bfu},\bfk\rbra_{\!\gamma(t)}
 = 0 & \forall\bfk &\shin\HS, \\[-0.5ex]
\text{(d) \ }& \lbra \dd{\bfh}, \bfrh \rbra_{\!\gamma(t)}
 = 0 & \forall\bfrh&\shin\RS,
\end{aligned}\right.\label{der:weak}\hspace*{-1.5em}
\end{multlined}
\end{equation}
where the particle motion $\gamma(t)=\bfxr(\gamma_0,t)$ is known (from solving the forward problem). Moreover, $\dd{\bfx}$ satisfies the initial condition~\eqref{beta:init}. The (symmetric in $\lpar(\bfu,p),(\bfv,q)\rpar$) tensor-valued function $\bfE$ is defined by
\begin{equation}\begin{multlined}
  \bfE\lpar(\bfu,p),(\bfv,q)\rpar
 = 2(\bfD[\bfu]\dip\bfD[\bfv])\bfI - 2\bfD[\bfu]\sip\bfna\bfv - 2\bfD[\bfv]\sip\bfna\bfu \\
   + p \lsqb \bfna\bfv - (\Div\bfv) \bfI \rsqb + q \lsqb \bfna\bfu - (\Div\bfu) \bfI \rsqb, \label{E:expr}
\end{multlined}
\end{equation}
and the Lagrangian derivative $\dd{\bfu}{}\Dsup$ of the Dirichlet data $\bfuD$ involved in~(\ref{der:weak}b) is given by
\begin{equation}
  \dd{\bfu}{}\Dsup
 = \dd{\ell}\bftau + \ell(\del{s}\theta_n + \kappa\theta_s)\bfn, \qquad
  \text{with \ \ } \dd{\ell} = -\int_{0}^{\ell} \kappa\theta_n \ds. \label{lm4}
\end{equation}
\end{proposition}
\begin{proof}
The proposition is obtained by applying the material differentiation identities~\eqref{dd:IVS} to the weak formulation~\eqref{forward:weak}, assuming that the test functions satisfy $(\dd{\bfv},\dd{q},\dd{\bfg},\dd{\bfk},\dd{\bfrh})=(\bfze,0,\bfze,\bfze,\bfze)$, i.e. are convected under the shape perturbation. The latter provision is made possible by the absence of boundary constraints in the definition of $\US,\Pcal,\FS,\HS,\RS$ (Sec.~\ref{sec:forward:weak}). Moreover, equations (a), (c), (d) use that $\DivS\dd{\bfx}\sheq0$ (Lemma~\ref{DivS:theta}), while (c) also exploits~\eqref{dd:u:beta}. The tensor-valued function $\bfE$ arises from rearranging all domain integrals that explicitly involve either $\bfth$ or $\dd{\bfx}$. Finally, the proof of the given expression for $\dd{\bfu}{}\Dsup$ is deferred to Appendix~\ref{app}.\enlargethispage*{5ex}
\end{proof}

\begin{remark}
The provision $\bfth\shin C_0^{1,\infty}(\OO_{\text{all}})$ in~\eqref{Theta:def} ensures that domain integrals $\lbra \bfE,\bfna\Tsup\!\bfth\rbra_{\OO}$ appearing in Proposition~\ref{lm1} are well-defined for any $(\bfu,p)\shin\US\shtimes\Pcal$.
\end{remark}
\begin{remark}
The mean of $\dd{p}$ is in practice irrelevant; setting it through $\lbra \dd{p},1\rbra_{\OO}+\lbra p\Div\bfth,1\rbra_{\OO}=0$ would preserve the zero-mean constraint on $p$ under shape perturbations.
\end{remark}
\begin{remark}
The tensor-valued function $\bfE$ given by~\eqref{E:expr} is the analog for Stokes flows of the elastic energy-momentum tensor~\cite{eshelby:75}, which plays a central role in the analysis of energy changes induced by crack growth in solids.
\end{remark}

\subsection{Shape derivative of a generic functional}
\label{dd:J}

Consider generic objective functionals
\begin{equation}
  J(\G) = \ioT G(\bfx) \dV + \iT \Lcb \iG F(\bff,\G) \ds + \iGL H(u_1) \ds \Rcb \dt \label{J:gen:def}
\end{equation}
where $\bff$, $\bfu$ and $\oo(T)$ (through $\bfxr(\cdot,T)$) are components of the forward solution and $F,G,H$ are sufficiently regular densities. The dissipation, particle centroid and mass flow rate functionals introduced in Sec.~\ref{optim:pb} all have the format~\eqref{J:gen:def}, see Sec.~\ref{dd:explicit}, thanks in particular to the assumed explicit dependence of $F$ on the wall shape. The chosen notation $J(\G)$ serves to emphasize the fact that the shape dependency is driven by $\G$; in particular, the particle motion induces a $\G$-dependent evolution of the fluid domain $\OO(T)$.

The derivative of the cost functional~\eqref{J:gen:def} is then given, using~(\ref{dd:IVS}a,b), by
\begin{equation}
\begin{multlined}\label{dJ:gen}
  J'(\G;\bfth)
 = \ioT \bfna G\sip\dd{\bfx} \dV + \iT \Lcb \iG \lsqb \del{\bff}F(\bff,\G)\sip\dd{\bff} + F^1(\bff,\G,\bfth) + F(\bff,\G)\DivS\bfth \rsqb \ds \\
  + \iGL \lsqb \del{u_1}H(u_1)\dd{u}_1 + H(u_1)\del{2}\theta_2 \rsqb \ds \Rcb \dt,
\end{multlined}
\end{equation}
where $F^1:=\dd{F}|_{\scriptsize \dd{\bff}=\bfze}$ and having used that $\DivS\bfth=\del{2}\theta_2$ on $\GL$ and $\Div\dd{\bfx}\sheq0$ in $\oo(t)$.

\subsubsection*{Adjoint problem} The shape derivative $J'(\G)$ involves the forward solution derivatives $\dd{\bfu},\dd{\bff},\dd{\bfx}$ solving problem~\eqref{der:weak}. Finding the latter therefore seems at first glance necessary for evaluating $J'(\G;\bfth)$ in a given shape perturbation $\bfth$, but in fact can be avoided with the help of an adjoint problem defined at any time $t$ by the weak formulation:
\begin{align}
\lefteqn{
\text{For $t\shin[0,T]$, find }(\bfuh,\hatp,\bffh,\bfhh,\bfxh) \in \US\shtimes\Pcal\shtimes\FS\shtimes\HS\shtimes\RS,  } & \suite
\left\{
\begin{aligned}
\text{(a) \ \ }&	a(\bfv,\bfuh) -b(\bfuh,q) - b(\bfv,\hatp) - \lbra\bfv,\bffh\rbra_{\G}
 - \lbra\bfv,\bfhh\rbra_{\!\gamma(t)}
 = -\lbra \del{u_1}H,v_1 \rbra_{\G_L} & \forall(\bfv,q) &\shin\US\shtimes\Pcal, \\
\text{(b) \ \ }& \lbra\bfg,\bfuh\rbra_{\G}
 = \lbra \del{\bff}F,\bfg \rbra_{\G} & \forall\bfg &\shin\FS, \\
\text{(c) \ \ }& - \lbra\bfk,\bfuh\rbra_{\!\gamma(t)} + \lbra \bfk, \bfxh \rbra_{\!\gamma(t)}
 = 0 & \forall\bfk &\shin\HS, \\[-1ex]
\text{(d) \ \ }& -\lbra \bdot{\bfhh},\bfrh \rbra_{\!\gamma(t)}
 + \lbra\bfE\lpar(\bfu,p),(\bfuh,\hatp)\rpar,\bfna\Tsup\!\bfrh \rbra_{\OO(t)} = 0 & \forall\bfrh &\shin\RS,  
\end{aligned}\right. \label{adj:weak} \intertext{and the final condition}
 \lefteqn{\qquad\text{(e) \ \ } \lbra \bfhh(\cdot,T),\bfrh \rbra_{\!\gamma(T)} = -\lbra \bfna G,\bfrh \rbra_{\!\omega(T)}
 \qquad \text{at $t\sheq T$}\hspace*{12.5em} \forall\bfrh\shin\RS,} \label{final:cond}
\end{align}
where the particle motion $\oo(t)$ is again known from solving the forward problem.
The adjoint state $(\bfuh,\hatp,\bffh,\bfhh,\bfxh)$ is thus created by applying a pressure difference $\Delta\hatp=\del{u_1}H$ between the channel end sections, while prescribing a velocity $\bfuh=\del{\bff}G$ on the channel walls; moreover, condition~(\ref{adj:weak}d) links the evolution of the net hydrodynamic force and torque on $\gamma(t)$ to the other variables of the adjoint solution. The particle derivative $\bdot{\bfhh}$ of the adjoint traction $\hat{\bfh}$ is taken following the known motion of the particle $\oo(t)$.\enlargethispage*{3ex}

A backward time-stepping treatment using the sequence of discrete times introduced in Sec.~\ref{sec:forward:weak} may be defined by treating the particle derivative $\bdot{\bfhh}$ in Euler-explicit form, setting $\bfhh_n\shdeq\bfhh(\bfx(t_n),t_{n})$ (i.e. following material points $\bfx_n$ in the known forward motion of $\gamma$) and $\bdot{\bfhh}_{n+1}\approx\lpar \bfhh_{n+1}\shm\bfhh_{n} \rpar/\Delta t$. Condition~(\ref{adj:weak}d) then takes the form
\begin{equation}\label{adj:weak:e:FD}
  \lbra \bfhh_{n},\bfrh \rbra_{\!\gamma_{n+1}}
 = \lbra \bfhh_{n+1},\bfrh \rbra_{\!\gamma_{n+1}} - \Delta t \lbra \bfn\sip\bfE_{n+1},\bfrh \rbra_{\!\gamma_{n+1}}
\end{equation}
where $\bfn\sip\bfE_{n+1}$ is given by \eqref{E:gamma} with the forward and adjoint solutions evaluated at $t\sheq t_{n+1}$. A natural time-stepping method for the adjoint problem then is:\enlargethispage*{1ex}
\begin{enumerate}
\item Final time ($t\sheq t_N$): solve equations~(\ref{adj:weak}a-c) and~\eqref{final:cond} for $\lpar \bfuh,\hatp,\bffh,\bfhh,\bfxh \rpar(t_N)$.
\item Generic time ($t\sheq t_n,\;0\shleq n\shl N$): solve equations~(\ref{adj:weak}a-d) for $\lpar\bfuh,\hatp,\bffh,\bfhh,\bfxh\rpar(t_n)$, with condition~(\ref{adj:weak}d) in the time-discrete form~\eqref{adj:weak:e:FD}.\vspace*{2ex}
\end{enumerate}
\begin{remark}
Like the forward problem~\eqref{forward:weak}, the adjoint problem~\eqref{adj:weak} is evolutive. The adjoint solution evolves backwards in time, from the final condition~(\ref{adj:weak}e). The particle motion in problem~\eqref{adj:weak} is given, whereas it was unknown in problem~\eqref{forward:weak}.
\end{remark}
\begin{remark}
The provision $\forall\bfrh\shin\RS$ in~(\ref{adj:weak}d) is a notational abuse, as $\bfrh$ therein is an extension to $C^{1,\infty}(\overline{\OO})$ of a rigid-body transformation velocity $\bfrh|_{\!\oo(t)}\shin\RS$. Lemma~\ref{lm5} will show that, for given $\bfrh|_{\!\oo(t)}\shin\RS$, (\ref{adj:weak}d) does not depend on the choice of extension.
\end{remark}

\subsubsection*{Shape derivative using adjoint solution} Now, combining the derivative problem~\eqref{der:weak} and the adjoint problem~\eqref{adj:weak} with appropriate choices of test functions, we obtain an expression of $J'(\G;\bfth)$ that no longer involves the derivative solution:

\begin{lemma}\label{lm2}
The shape derivative $J'(\G;\bfth)$ is given by
\begin{equation}
\begin{multlined}\label{dd:J:gen:lm2}
  J'(\G;\bfth)
 = \iT \Lcb \lbra\bfE\lpar(\bfu,p),(\bfuh,\hatp)\rpar,\bfna\Tsup\!\bfth \rbra_{\OO(t)}
  + \lbra\dd{\bfu}{}\Dsup,\bffh\rbra_{\G} + \lbra H(u_1),\del{2}\theta_2 \rbra_{\G_L} \hspace*{-1ex} \\
  + \iG \lcb F^1(\bff,\G,\bfth) + \lsqb F(\bff,\G) \shm \del{\bff}F(\bff,\G) \sip \bff \rsqb\,\DivS\bfth \rcb \ds \Rcb\dt
\end{multlined}
\end{equation}
in terms of the transformation velocity $\bfth$ on $\G$, of $(\bfu,p,\bff)$ solving the forward problem~\eqref{forward:weak}, and of $(\bfuh,\hatp,\bffh)$ solving the adjoint problem~\eqref{adj:weak}.
\end{lemma}
\begin{proof}
The test functions $(\bfv,q,\bfg,\bfk,\bfrh)$ are set to $(\bfuh,\hatp,\bffh,\bfhh,\bfxh)$ in the derivative problem~\eqref{der:weak} and to $(\dd{\bfu},\dd{p},\dd{\bff},\dd{\bfh},\dd{\bfx})$ in the adjoint problem~\eqref{adj:weak}, and the combination $(\ref{der:weak}a)+(\ref{der:weak}b)+(\ref{der:weak}c)+(\ref{der:weak}d)-(\ref{adj:weak}a)+(\ref{adj:weak}b)-(\ref{adj:weak}c)-(\ref{adj:weak}d)$ then evaluated (using $\bfuh=\del{\bff}G$ on $\G$, implied by~(\ref{adj:weak}b), along the way). This results in
\begin{multline} \label{aux4}
  \lbra \del{\bff}F,\dd{\bff} \rbra_{\G} + \lbra \del{u_1}H,\dd{u_1} \rbra_{\G_L} 
 = \lbra\bfE\lpar(\bfu,p),(\bfuh,\hatp)\rpar,\bfna\Tsup\!\bfth \rbra_{\OO(t)} + \lbra\dd{\bfu}{}\Dsup,\bffh\rbra_{\G} \\[-1ex]
  - \lbra \del{\bff}F,\bff\DivS\bfth \rbra_{\G}
  + \lbra \dd{\bfx},\bdot{\bfhh}\rbra_{\!\gamma(t)} + \lbra (\dd{\bfx}^{\Bdot}\!\!\!),\bfhh\rbra_{\!\gamma(t)},
\end{multline}
which we then use in expression~\eqref{dJ:gen} of $J'(\G;\bfth)$ to obtain
\begin{align}
  J'(\G;\bfth)
 &= \lbra \bfna G,\dd{\bfx} \rbra_{\!\omega(T)}
  + \iT \Lcb \lbra\bfE\lpar(\bfu,p),(\bfuh,\hatp)\rpar,\bfna\Tsup\!\bfth \rbra_{\OO(t)}
  + \lbra\dd{\bfu}{}\Dsup,\bffh\rbra_{\G} + \lbra H(u_1),\del{2}\theta_2 \rbra_{\G_L} \suite\quad
  + \iG \lcb F^1(\bff,\G,\bfth) + \lsqb F(\bff,\G) \shm \del{\bff}F \sip \bff \rsqb\,\DivS\bfth \rcb \ds
  + \lbra \dd{\bfx},\bdot{\bfhh}\rbra_{\!\gamma(t)} + \lbra(\dd{\bfx}^{\Bdot}\!\!\!),\bfhh\rbra_{\!\gamma(t)} \Rcb \dt
\label{dJ:gen:2}
\end{align}
Then, we observe that the last two terms in the above formula combine to an exact particle time derivative (by virtue of the differentiation identity~(\ref{dd:IVS}b) wherein $\eta$ and $\bfth$ are replaced with the physical time $t$ and particle velocity $\bdot{\bfx}$, and recalling that $\DivS\bdot{\bfx}\sheq0$):
\begin{equation}
\begin{multlined} \label{aux5}
   \iT \Lcb \lbra 
   \lbra \dd{\bfx},\bdot{\bfhh}\rbra_{\!\gamma(t)} + (\dd{\bfx}^{\Bdot}\!\!\!),\bfhh\rbra_{\!\gamma(t)} \rcb \dt 
 = \iT \der{}{t}\lbra \dd{\bfx},\bfhh\rbra_{\!\gamma(t)} \dt
 = \lbra \dd{\bfx},\bfhh\rbra_{\!\gamma} \Rabs^{t=T}_{t=0}
 = -\lbra \bfna G,\dd{\bfx} \rbra_{\!\omega(T)},
\end{multlined}
\end{equation}
with the last equality resulting from the initial condition~(\ref{der:weak}e) and the final condition~(\ref{adj:weak}e). As a result, \eqref{dJ:gen:2} yields $J'(\G;\bfth)$ as claimed in the Lemma.
\end{proof}

\begin{remark}
The evolution equation~(\ref{adj:weak}d) and final condition~(\ref{adj:weak}e) are designed to achieve complete elimination from $J'(\G;\bfth)$ of the induced transformation velocity $\dd{\bfx}$ (featured among the unknowns of the derivative problem~\eqref{der:weak}); as a result (and as usual), the adjoint solution evolves backwards in time. We moreover observe that Lemma~\ref{lm2} crucially exploits the weak forms of the derivative and adjoint problems.\enlargethispage*{1ex}
\end{remark}

\subsubsection*{Boundary-only formulation of the shape derivative} Neither the adjoint problem~\eqref{adj:weak} nor the shape derivative expression provided by Lemma~\ref{lm2} can be directly used within a BIE framework, in both cases because of the domain integral terms involving $\bfE$. We now show that those terms can be reformulated as boundary integrals involving only quantities defined on $\G$ and $\GL$, thanks to the following identity:
\begin{lemma}\label{lm5}
Let $(\bfu,p)$ and $(\bfuh,\hatp)$ respectively satisfy $\Div\bfu\sheq0$, $-\Delta\bfu\shp\bfna p\sheq\bfze$ and $\Div\bfuh\sheq0$, $-\Delta\bfuh\shp\bfna\hatp\sheq\bfze$ in $\OO$. Assume that $\bfu$, $\bfuh$ and $p$ are periodic, and set $\Delta\hatp(x_2) := \hatp(L,x_2)-\hatp(0,x_2)$ (i.e. periodicity is not assumed for $\hatp$). Then, for any vector field $\bfzet\in C^{1,\infty}_0(\OO_{\text{all}})$, the following identity holds:
\begin{equation}
  \lbra\bfE\lpar(\bfu,p),(\bfuh,\hatp)\rpar,\bfna\Tsup\!\bfzet \rbra_{\OO(t)}
 = \int_{\G\cup\gamma(t)} \bfn\sip\bfE\lpar(\bfu,p),(\bfuh,\hatp)\rpar\sip\bfzet \ds + \int_{\G_L} \Delta\hatp \, (\del{2} u_1)\zeta_2 \ds. 
\end{equation}
Moreover, if the traces on $\gamma$ of $\bfu,\bfuh$ are rigid-body velocities with respective angular velocities $\varrho,\ooh$, we have
\begin{equation}
  \bfn\sip\bfE\lpar(\bfu,p),(\bfuh,\hatp)\rpar
 = -\lpar \ooh\bfh \shp \varrho\bfhh \rpar\sip\bfr - h_s\hath_s\,\bfn \qquad\text{on $\gamma$}, \label{E:gamma}
\end{equation}
where $\bfh\sheq\bfsig[\bfu,p]\sip\bfn$, $\bfhh\sheq\bfsig[\bfuh,\hatp]\sip\bfn$ and $\bfr=\bfe_2\tensor\bfe_1-\bfe_1\tensor\bfe_2=\bfn\shtimes\bftau-\bftau\shtimes\bfn$ (see~\eqref{rigid:velocity}).
\end{lemma}
\begin{proof}
See Appendix~\ref{lm5:proof}.
\end{proof}

Lemma~\ref{lm5} is first applied, with $\bfzet\sheq\bfrh$, to the term $\lbra\bfE\lpar(\bfu,p),(\bfuh,\hatp)\rpar,\bfna\Tsup\!\bfrh \rbra_{\OO(t)}$ in the adjoint evolution equation~(\ref{adj:weak}d), in which case the velocity fields $\bfu$ and $\bfuh$ both have rigid-body traces on $\gamma(t)$ while $\bfrh$ can be safely assumed to verify $\text{supp}(\bfrh)\cap\G=\emptyset$. The evolution equation~(\ref{adj:weak}d) thus becomes
\begin{equation}
  -\lbra \bdot{\bfhh},\bfrh \rbra_{\!\gamma(t)}
 + \lbra \bfn\sip\bfE\lpar(\bfu,p),(\bfuh,\hatp)\rpar,\bfrh \rbra_{\gamma(t)} = 0 \qquad \forall\bfrh\shin\RS, \label{evol:weak}
\end{equation}
with $\bfn\sip\bfE$ given by~\eqref{E:gamma}, allowing the adjoint problem~\eqref{adj:weak} to be recast in BIE form.\enlargethispage*{1ex}

We then evaluate $\lbra\bfE\lpar(\bfu,p),(\bfuh,\hatp)\rpar,\bfna\Tsup\!\bfth \rbra_{\OO(t)}$ in expression~\eqref{dd:J:gen:lm2} of $J'(\G;\bfth)$ by means of Lemma~\ref{lm5} applied (with $\bfzet\sheq\bfrh$) to the solutions of the forward problem~\eqref{forward:weak} and the adjoint problem~\eqref{adj:weak} (for which $\Delta \hatp=\del{u_1}H$). Observing along the way that
\begin{equation}
  \int_{\G_L} \del{u_1}H(u_1)\; (\del{2} u_1)\theta_2 \ds
 + \int_{\G_L} H(u_1) \, \del{2}\theta_2 \ds
 = \iGL \del{2}\lpar H(u_1)\theta_2 \rpar \ds,
\end{equation}
the shape derivative of $J$ is recast in the following form, without domain integrals:
\begin{align}
  J'(\G;\bfth)
 &= \iT \Lcb \iG \Lpar \bffh\sip\dd{\bfu}{}\Dsup + \bfn\sip\bfE\lpar(\bfu,p),(\bfuh,\hatp)\rpar\sip\bfth \Rpar \ds
  + \iGL \del{2}\lpar H(u_1)\theta_2 \rpar \ds \suite\qquad 
  + \iG \Lpar F^1(\bff,\G,\bfth) + \lsqb F(\bff,\G) \shm \del{\bff}F(\bff,\G) \sip \bff \rsqb\,\DivS\bfth \Rpar \ds \Rcb \dt \label{dd:J:gen:G}
\end{align}

Expression~\eqref{dd:J:gen:G} is still somewhat inconvenient for use in a BIE framework as it involves (through $\bfD[\bfu]$ and $\bfD[\bfuh]$ in $\bfE$) the complete velocity gradient on $\G$. This can be alleviated by reformulating the latter in terms of tractions and tangential derivatives of velocities, eliminating normal derivatives of velocities by means of the constitutive relation~(\ref{forward:PDE}b). This step is here implemented through the following explicit auxiliary identity, established (in Appendices~\ref{grad:curv} and~\ref{proof:lemma5}) using curvilinear coordinates:
\begin{align}
\lefteqn{
  \bffh\sip\dd{\bfu}{}\Dsup + \bfn\sip\bfE\lpar(\bfu,p),(\bfuh,\hatp)\rpar\sip\bfth
 + \lsqb F(\bff,\G) \shm \del{\bff}F(\bff,\G) \sip \bff \rsqb\,\DivS\bfth }
 & \label{forward:wall} \\[-0.5ex] & \mbox{}\hspace{5pt}
 = d_s \lpar \lsqb F(\bff,\G) \shm \del{\bff}F(\bff,\G) \sip \bff \rsqb\theta_s \rpar
  + \dd{\ell}\fhat_s - \ell(\del{s}\theta_n)\fhat_n
 + \kappa\ell \fhat_s\theta_n - \lpar\del{s} F\rpar\theta_s
  \suite\quad
  + \Lpar \lsqb f_s\bfn - p\bftau \rsqb\sip\del{s}\bfuh - f_s\fhat_s
  - \kappa\lsqb F(\bff,\G) \shm \del{\bff}F(\bff,\G) \sip \bff \rsqb \Rpar\theta_n,
\end{align}
where $\del{s} F$ indicates the partial derivative w.r.t. $s$ of $F(\bff,\G)$ (with $\bff$ frozen) while $\text{d}_s$ denotes a total derivative w.r.t. $s$.\enlargethispage*{1ex}
We now use the above identities into~\eqref{dd:J:gen:G}. Since $\text{d}_s \lpar \lsqb F \shm \del{\bff}F\sip\bff \rsqb\theta_s \rpar \ds$ integrates to zero over $\G$ by virtue of the spatial periodicity of the forward solution and requirement (i) of~\eqref{Theta:def}, we obtain the following final result for $J'(\G;\bfth)$, suitable for direct implementation using the output of a BIE solver:
\begin{proposition}\label{dJ:final}
The shape derivative of any cost functional $J$ of the form~\eqref{J:gen:def} in a shape perturbation whose transformation velocity field $\bfth$ satisfies assumptions~\eqref{Theta:def} is given (with $f_s\shdeq\bff\sip\bftau$, $\fhat_s\shdeq\bffh\sip\bftau$) by
\begin{align}
  J'(\G;\bfth)
 &= \iT \Lcb \iG \Lpar F^1 - \lpar\del{s}F\rpar\theta_s
    + \dd{\ell} \fhat_s - \ell(\del{s}\theta_n)(\hatp\shp2\bftau\sip\del{s}\bfuh) \Rpar \ds
      + \iGL \del{2}\lpar H(u_1)\theta_2 \rpar \ds \suite\qquad
    + \iG \Lpar (f_s\bfn \shm p\bftau)\sip\del{s}\bfuh + \kappa\ell\fhat_s - f_s\fhat_s
    - \kappa\lsqb F \shm \del{\bff}F \sip \bff \rsqb \Rpar \theta_n \ds.
\end{align}
\end{proposition}
We now apply Proposition~\ref{dJ:final} to the specific functionals introduced in Section~\ref{optim:pb}.

\subsection{Sensitivity results for functionals involved in pumping problem}
\label{dd:explicit}

The adjoint state solving the weak formulation~\eqref{adj:weak} satisfies the incompressible Stokes equations with periodicity conditions
\begin{subequations}
\begin{equation}
  -\Delta\bfuh + \bfna\hatp = \bfze, \quad \Div\bfuh = 0 \quad \text{in $\OO(t)$}, \qquad \bfuh|_{\G_L} = \bfuh|_{\G_0}, \label{adjoint:PDE}
\end{equation}
the fluid domain $\OO(t)$ and particle configuration $\gamma(t)$ being those determined by the forward problem. Moreover, the adjoint fluid motion results from the velocity being prescribed by
\begin{equation}
  \bfuh = \bfxh \qquad \text{on $\gamma(t)$} \label{adjoint:PDE:particle}
\end{equation}
on the particle, and by 
\begin{equation}
  \bfuh = \del{\bff}F \qquad \text{on $\G$}, \label{adjoint:PDE:wall}
\end{equation}
on the wall, as well as the pressure drop being prescribed as
\begin{equation}
  \hatp\mid_{\G_L} - \hatp\mid_{\G_0} = \del{u_1}H. \label{adjoint:PDE:drop}
\end{equation}
Moreover, $\bfhh$ and $\bffh$ in the weak adjoint problem~\eqref{adj:weak} are the stress vectors arising from the enforcement (as equality constraints) of the BCs~\eqref{adjoint:PDE:particle} and~\eqref{adjoint:PDE:wall}; in particular, $\bfhh\sheq\bfsig[\bfuh,\hatp]\sip\bfn$ on $\gamma(t)$ and $\bffh\sheq\bfsig[\bfuh,\hatp]\sip\bfn$ on $\G$.

Equations~(\ref{adjoint:PDE}-d) are the strong-form counterparts of equations~(\ref{adj:weak}a-c), and define a well-posed problem in case $\bfxh$ is given. However, like $\bfxr$ in the forward problem, $\bfxh$ is unknown. This is compensated by the fact that $\bfhh$ must satisfy additional requirements, namely the evolution equation~(\ref{adj:weak}d) and the final condition~(\ref{adj:weak}e). The strong form of the evolution equation is
\begin{equation}
\begin{aligned}
  \lbra \bdot{\bfhh} ,1 \rbra_{\gamma(t)}
 &= -\lbra \lpar\varrho\bfhh\sip\bfr \shp h_s\hath_s\,\bfn\rpar\,,1\rbra_{\gamma(t)}, \\[-1ex]
  \lbra \bdot{\bfhh},\bfe_3\shtimes\bfx \rbra_{\gamma(t)}
 &= \lbra \lpar\ooh\bfh \shp \varrho\bfhh \shm h_s\hath_s\,\bftau\rpar\,,\bfx \rbra_{\gamma(t)}
\end{aligned} \qquad t\shin[0,T]. \label{adjoint:evolution}
\end{equation}
(having invoked~\eqref{E:gamma} and used that $\bfr\sip(\bfe_3\shtimes\bfx)=-\bfx$, $\bfn\sip(\bfe_3\shtimes\bfx)=\bftau\sip\bfx$ and, by virtue of~(\ref{forward:weak}d), $\lbra \bfh,1\rbra_{\gamma}\sheq0$) for the evolution equation, while that of the final condition reads
\begin{equation}
\begin{aligned}
  \lbra \bfhh(\cdot,T),1 \rbra_{\gamma(T)} &= - \lbra \bfna G,1 \rbra_{\oo(T)}, \\
  \lbra \bfhh(\cdot,T),\bfe_3\shtimes\bfx \rbra_{\gamma(T)} &= -\lbra \bfna G, \bfe_3\shtimes\bfx \rbra_{\oo(T)}.
\end{aligned}
\label{adjoint:final}
\end{equation}
\end{subequations}
Equations~(\ref{adjoint:PDE}-f) together constitute the strong form of the weak adjoint problem~\eqref{adj:weak}. Equations~\eqref{adjoint:PDE:particle}, \eqref{adjoint:PDE:drop} and~\eqref{adjoint:final} depend on the objective function being considered, whereas equations~(\ref{adjoint:PDE},b,e) do not.

\subsubsection*{Shape derivative of dissipation functional}

The dissipation functional $J\Wsub$, defined by~\eqref{J:loss:def}, is a particular instance of~\eqref{J:gen:def} with $F(\bff,\G)=(\bfuD\shp \bfe_1)\sip\bff$ and $G\sheq H\sheq0$. In particular, we have
\begin{equation}
  \del{s} F =\lpar\del{s}\bfuD\rpar\sip\bff = -\kappa\ell p, \qquad
  \del{\bff}F\sheq\bfuD\shp \bfe_1,
\end{equation}
from which we find $\del{s}\bfuh = (\kappa\ell)\bfn$. We also have $F \shm \bff\sip\del{\bff}F=0$. Finally, $F$ depends on $\G$ through $\bfuD$ given by~\eqref{noslip:2D}, so that $F^1=\dd{\bfu}{}\Dsup\sip\bff$, which in turn yields
\begin{equation}
  F^1\shm\del{s} F
 = \dd{\ell}f_s - \kappa\ell\del{s}\theta_n
\end{equation}
with the help of~\eqref{lm4}. The adjoint solution is governed in strong form by equations~\eqref{adjoint:PDE} to~\eqref{adjoint:final} particularized to the case of $J\Wsub$, i.e. the velocity on the particle and the pressure drop are prescribed as
\begin{equation}
  \bfuh = \bfu\Dsup\shp \bfe_1 \quad \text{on $\G$}, \qquad \hatp\mid_{\G_L}\shm\hatp\mid_{\G_0} = 0.
\end{equation}
while the final conditions~\eqref{adjoint:final} on $\bfhh$ are (since $G\sheq0$) homogeneous:
\begin{equation}
  \lbra \bfhh(\cdot,T),1 \rbra_{\gamma(T)} = 0, \qquad
  \lbra \bfhh(\cdot,T),\bfe_3\shtimes\bfx \rbra_{\gamma(T)} = 0 \label{adjoint:final:hom}
\end{equation}
Applying Proposition~\ref{dJ:final} to this case, the shape derivative of $J\Wsub$ is therefore obtained (upon evaluation of $\dd{\bfu}{}\Dsup\sip\bff$ with~\eqref{lm4}) as\enlargethispage*{1ex}
\begin{equation} \label{dd:Jw}
  J'\Wsub(\G;\bfth)
 = \iT \!\! \iG \Lcb \Lsqb \kappa\ell(f_s\shp\fhat_s) - f_s\fhat_s \Rsqb\theta_n
     + \dd{\ell}(f_s\shp\fhat_s) - \ell(\del{s}\theta_n)(p\shp\hatp) \Rcb \ds \dt 
\end{equation}

\subsubsection*{Shape derivative of net particle motion} The net particle motion $D(\G)$, defined by~\eqref{xG:def}. is another shape functional of the form~\eqref{J:gen:def}, with $G(\bfx)= x_1/|\oo|$ and $F\sheq H\sheq0$. The adjoint problem in strong form still consists of equations~\eqref{adjoint:PDE} to~\eqref{adjoint:final}, whose particularization for $D(\G)$ results in vanishing entails setting to zero the velocity on the particle and the pressure:
\begin{equation}
  \bfuh = \bfze \quad \text{on $\G$}, \qquad \hatp\mid_{\G_L}\shm\hatp\mid_{\G_0} = 0,
\end{equation}
while the final conditions~\eqref{adjoint:final} become
\begin{equation}
  \text{(a) \ }\lbra \bfhh(\cdot,T),1 \rbra_{\gamma(T)} = -\bfe_1, \qquad 
  \text{(b) \ }\lbra \bfhh(\cdot,T),\bfe_3\shtimes\bfx \rbra_{\gamma(T)} = x_2\Gsup(T) \label{adjoint:final:centroid}
\end{equation}
We note that conditions~\eqref{adjoint:evolution} and~\eqref{adjoint:final:centroid}, as well as the definition of $\bfhh$ as a traction vector, assume the orientation convention of Fig.~\ref{geom:2D} on $\gamma(t)$ while $\bfx$ is the absolute vector position in~(\ref{adjoint:final:centroid}b). The derivative $D'(\G;\bfth)$ of $D(\G)$ is found from Proposition~\ref{dJ:final} to be given by the right-hand side of~\eqref{dd:cQ:result} without the contributions of $\G_L$, i.e.:\enlargethispage*{1ex}
\begin{equation}
  D'(\G;\bfth)
 = \int_0^{T}\!\! \iG \Lcb \lpar \ell\kappa\fhat_s - f_s\fhat_s \rpar \theta_n
     + \dd{\ell} \fhat_s - \ell(\del{s}\theta_n)\hatp \Rcb \ds \dt \label{dd:D}
\end{equation}

\subsubsection*{Shape derivative of mass flow rate functional}

The shape derivative of the time-averaged mass flow functional $C(\G)$ defined by~\eqref{flowrate:expr} is given by
\begin{equation}
  Q'(\G;\bfth)
 = - |\oo|\lbra D'(\G),\bfth \rbra + \lbra |\OO|',\bfth \rbra + \lbra C'(\G),\bfth \rbra,
\end{equation}
with the first two derivatives respectively given by~\eqref{dd:D} and~\eqref{dd:volume}, so that we only need to focus on the evaluation of $C'(\G;\bfth)$. $C(\G)$, defined in~\eqref{flowrate:expr}, is a shape functional of the form~\eqref{J:gen:def}, with $F\sheq G\sheq 0$ and $H(u_1)\sheq u_1/T$. The adjoint solution associated with $C(\G)$ therefore solves problem~\eqref{adjoint:PDE}--\eqref{adjoint:final} with the above-specified $F,G,H$, so that~(\ref{adjoint:PDE:particle},d) become
\begin{equation}
  \bfuh = \bfze \quad \text{on $\G$}, \qquad \hatp\mid_{\G_L}\shm\hatp\mid_{\G_0} = 1/T
\end{equation}
and the homogeneous final conditions~\eqref{adjoint:final:hom} again apply.
The shape derivative of $C(\G)$ is finally found from Proposition~\ref{dJ:final} to be given by
\begin{align} \label{dd:cQ:result}
  C'(\G;\bfth)
 &= \iT \!\! \iG \Lcb \lpar \ell\kappa\fhat_s - f_s\fhat_s \rpar \theta_n + \dd{\ell} \fhat_s
    - \ell(\del{s}\theta_n)\hatp \Rcb \ds \dt \suite\qquad\qquad
   + \inv{T}\iT \lsqb \lpar u_1\theta_2 \rpar(\bfz^+,\cdot) - \lpar u_1\theta_2 \rpar(\bfz^-,\cdot) \rsqb \dt
\end{align}

\section{Numerical scheme}\label{sc:numerical}
In this section, we describe our numerical solvers for the shape optimization problem~\eqref{eq:opt} that employ the shape sensitivity formulas derived in the previous section. 
\subsection{Optimization method}\label{sec:optmethod}

To avoid second-order derivatives of the cost functional, whose evaluation is somewhat challenging in our case, we solve the shape optimization problem~\eqref{eq:opt} using an augmented Lagrangian (AL) approach and Broyden-Fletcher-Goldfarb-Shanno (BFGS) algorithm. 
An augmented Lagrangian $\Lcal\Asub$ is defined by
\begin{equation}
  \Lcal\Asub(\OO,\lambda;\sigma)
 = J\Wsub(\OO) - \lambda_1 C\Vsub(\OO)  - \lambda_2 C\Dsub(\OO)+ \frac{\sigma}{2} [ C^2\Vsub(\OO)+C^2\Dsub(\OO)],
\label{eq:augla}\end{equation}
where $\sigma$ is a positive penalty coefficient and  $\lambda = (\lambda_1, \lambda_2)$ are Lagrange multipliers. Setting the initial values $\sigma^0$ and $\lambda^0$ using heuristics, the AL method introduces a sequence $(m= 1, 2,\dots)$ of unconstrained minimization problems:
\begin{equation}
 \OO_m =  \argmin_{\OO \, \shin\, \Ocal}  \Lcal\Asub(\OO,\lambda^m;\sigma^m) ,
\label{eq:minsubp}\end{equation}
with explicit Lagrange multiplier estimates $\lambda^m$ and increasing penalties $\sigma^m$.
We use the BFGS algorithm \cite{NoceWrig06}, a quasi-Newton method, for solving~\eqref{eq:minsubp}.  Equations~\eqref{dd:volume}, \eqref{dd:Jw} and \eqref{dd:D} are used in this context for gradient evaluations in the line search method. The overall optimization procedure for problem~\eqref{eq:opt} is summarized in the following algorithm:

\noindent\rule{\textwidth}{0.5pt}
\begin{enumerate}
    \item[1:] Choose initial fluid region $\OO_0$
    \item[2:] Set convergence tolerance $\zeta^{\star}$, $\lambda^0$, $\sigma^0$, and $\zeta^1= (\sigma^0)^{-0.1}$
    \item[3:] \textbf{for} $m=1,2, \dots$ \textbf{do}
    \begin{enumerate}
        \item[(3-a):] Solve unconstrained minimization~\eqref{eq:minsubp} for $\OO_m$, {go to} (3-b)
        
        \item[(3-b):] \textbf{if} $\max(|C\Vsub(\OO_m)|,|C\Dsub(\OO_m)|) \shl \zeta^{m}$ \textbf{then} {go to} (3-c), \textbf{else} {go to} (3-e)
        
        \item[(3-c):] \textbf{if} $\max(|C\Vsub(\OO_m)|,|C\Dsub(\OO_m)|) \shl \zeta^{\star}$  \textbf{then} STOP and \textbf{return} $\OO^{\star}:=\OO_m$, \textbf{else} {go to} (3-d)
        
        \item[(3-d):] 
        \# update multiplier\\ $\lambda^{m}_1=\lambda^{m-1}_1-\sigma^{m-1} C\Vsub(\OO_m)$, \; $\lambda^{m}_2=\lambda^{m-1}_2-\sigma^{m-1} C\Dsub(\OO_m)$ \\
        $\sigma^{m}=\sigma^{m-1}$, $\zeta^{m+1}=(\sigma^m)^{-0.9}\zeta^m$  \\
        {go to} (3-a)

        \item[(3-e):] \# increase penalty \\ $\sigma^{m}=10\sigma^{m-1}$ \\
        $\lambda^{m}_1=\lambda^{m-1}_1, \lambda^{m}_2=\lambda^{m-1}_2$,  $\zeta^{m+1}=(\sigma^m)^{-0.1}$ \;   \\
        {go to} (3-a)
    \end{enumerate} 
\end{enumerate}
\noindent\rule{\textwidth}{0.5pt}

\subsection{Finite-dimensional parametrization of wall shapes}

We model the shape of the channel walls using B-splines. For an integer $k$, the  $k$-th cardinal B-spline basis function of degree $n$, denoted by $B_{k,n}$, is given by recurrence,
\begin{equation}
\begin{aligned}
\mathcal{B}_{k,0}(t) & = & \left\{\begin{array}{ll} 1, & k\le t < k\shp 1 \\ 0,& \mbox{otherwise}  \end{array} \right.\\
\mathcal{B}_{k,n} (t) & = & \frac{t-k}{n}\mathcal{B}_{k,n\shm 1} (t) + \frac{n\shp k\shp 1\shm t}{n} \mathcal{B}_{k+1,n-1}(t) 
\end{aligned}\label{eq:Bn}
\end{equation}
and $\mathcal{B}_{k,n}(t)$ has support $[k, k\shp n\shp 1]$.
Any $\mathcal{C}^{n-1}$ function $x(t)$ defined on $[0, M]$ with $M$ being a positive integer can be approximated by a linear combination of the form $ x(t) = \sum_{k=-n}^{M-1} \xi_{k} \mathcal{B}_{k,n}(t)$ with $\xi_{k}\in \mathbb{R}$ and $t\in [0,M]$. In this work, we use B-splines of degree $5$, i.e., $n=5$ in \eqref{eq:Bn}. To parametrize wall shapes $\boldsymbol{x}(t)$ for $t\in[0,2\pi]$, we define the basis functions $B_k(t) = \mathcal{B}_{k,5}(\frac{M}{2\pi} t)$, where $M$ is a pre-assigned positive integer of the discretization. The wall $\Gamma^{\pm}\ni\bfx^{\pm} = \bfx^{\pm}(t;\bfxi)$ is then written as
\begin{equation}\label{eq:bspline1}\left.
\begin{aligned}
& x_1^{\pm}(t) = x_1^{\pm}(t;\bfxi) = \frac{L}{2\pi}(2\pi\shm t) + \sum_{k=-5}^{M-1} \xi_{1,k}^{\pm} B_k (t), \\
& x_2^{\pm}(t) = x_2^{\pm}(t;\bfxi) = \phantom{(2\pi\shm t) + {}}\sum_{k=-5}^{M-1} \xi_{2,k}^{\pm} B_k (t),
\end{aligned}\right\} \quad  t \in [0, 2\pi],
\end{equation}
where $\bfxi$ is the vector of coefficients for $B_k$ with $(4M\shp 20)$ components. 
In the expression of $x_1^{\pm}(t)$, the extra term $\frac{L}{2\pi}(2\pi\shm t)$ ensures the periodicity of the linear combinations of B-splines which is enforced in the computation.

The domain $[0,2\pi]$ is divided into $M$ uniform subintervals and the corresponding endpoints $(x_1^{\pm}, x_2^{\pm})$ create a discretization grid for shape parametrization. 
We define the free discretization grid points by
\begin{equation}
  \bfps = \Lcb x_{1}^{\pm}\!\Lpar \frac{2\pi j}{M} \Rpar,\ x_{2}^{\pm}\!\Lpar \frac{2\pi j}{M} \Rpar, x_{2}^{\pm}(0)
\Rcb, j = 1,2,\dots, M\shm 1.
\end{equation}
The vector $\bfxi$ is then solved implicitly from equations~\eqref{eq:bspline1} (for given $\bfps$) together with the additional conditions
\begin{equation}\label{eq:bspline2}
x_{1}^{\pm}(0) = L, \ \ x_{1}^{\pm}(2\pi) = 0, \quad
x_{2}^{\pm}(0) = x_{2}^{\pm}(2\pi), \quad
\text{and} \qquad \frac{d^k \bfx^{\pm}}{d t^k}(0) =\frac{d^k \bfx^{\pm}}{d t^k}(2\pi), \;k=1,\dots, (n-1).
\end{equation}

The transformation velocities $\bfth$ are associated to perturbations of $\bfps$. Letting $\bfdel$ be a perturbation vector of the same dimension as $\bfps$ (i.e. with $(4M\shm 2)$ elements), the transformation velocities on both walls in the shape perturbation induced by $\bfdel$ are the limiting values of
\begin{equation}\label{eq:transformation}
\bfth^{\pm}(\bfx^{\pm}(t)) = \frac{1}{\eta} \left(  \bfx^{\pm}\left(t;\bfxi(\bfps + \eta \bfdel)\right) -  \bfx^{\pm}\left(t;\bfxi(\bfps)\right)  \right),
\end{equation}
as $\eta\to 0$. Since the mapping $\bfps\mapsto\bfxi(\bfps)$ is linear, it is unnecessary to actually take the limit in the above formula, and we simply use~\eqref{eq:transformation} with $\eta=1$ in the numerical implementation. In section \ref{sec:validation}, the wall shape perturbations are formulated by perturbing one element in $\bfps$ while keeping the others unchanged, so that $\bfdel$ is a vector with all $0$ except one unit entry.


\subsection{Boundary integral formulation}
\label{sec:forwardsolver}
The shape sensitivities require obtaining the traction and pressure on $\Gamma$ for the forward and associated adjoint problems. The fluid velocity and pressure in all these problems satisfy the Stokes equations with periodic boundary conditions. We follow the periodization scheme developed recently in \cite{marple:16} that uses the free-space Green's functions and enforces the periodic boundary conditions {\em via} an extended linear system approach. Given a source point $\bfy$ and a target point $\bfx$, the free-space Stokes single-layer and double-layer kernels are given by
\begin{equation}
	S_{ij}(\bfx,\bfy)=\frac{1}{4\pi}\left(\delta_{ij}\log\frac{1}{z}+\frac{z_iz_j}{z^2}\right), \quad D_{ij}(\bfx,\bfy)=\frac{1}{\pi}\frac{z_iz_j}{z^2}\frac{\bfz\cdot\bfn^{\bfy}}{z^2},
\end{equation}
where $\bfz:=\bfx-\bfy$, $z:=|\bfz|$. The associated pressure kernels are given by
\begin{equation}
	P^S_j(\bfx,\bfy)=\frac{1}{2\pi}\frac{z_j}{z^2},\quad P^D_j	(\bfx,\bfy)=\inv{\pi}\left(-\frac{n^{\bfy}_j}{z^2}+2\frac{\bfz\cdot\bfn^{\bfy}}{z^2}\frac{z_j}{z^2}\right),
\end{equation}
and the associated traction kernels are given by
\begin{equation}
	\begin{aligned}
		T^S_{ij}(\bfx,\bfy)& = -\frac{1}{\pi}\frac{z_iz_j}{z^2}\frac{\bfz\cdot\bfn^{\bfx}}{z^2},\\
		T^D_{ij}(\bfx,\bfy)& = \inv{\pi}\left[ \left(\frac{\bfn^{\bfy}\cdot\bfn^{\bfx}}{z^2}-8d_{\bfx}d_{\bfy} \right)\frac{z_iz_j}{z^2}+ d_{\bfx}d_{\bfy}\delta_{ij}+\frac{n_i^{\bfx}n_j^{\bfy}}{z^2}+d_{\bfx}\frac{z_jn_i^{\bfy}}{z^2}+d_{\bfy}\frac{z_in_j^{\bfx}}{z^2} \right],
	\end{aligned}
\end{equation}
where for notational convenience we defined the target and source ``dipole functions'' as
\begin{equation}
	d_{\bfx}=d_{\bfx}(\bfx,\bfy):=(\bfz\cdot \bfn^{\bfy})/z^2,\quad d_{\bfy}=d_{\bfy}(\bfx,\bfy):=(\bfz\cdot \bfn^{\bfx})/z^2.
\end{equation}
We employ an indirect integral equation formulation with the following ansatz:
\begin{equation}\label{eq:vel_rep}
	\bfu = \Dcal_{\Gamma}^{\text{near}}\bftau_{\Gamma} + \Scal_{\gamma}^{\text{near}}\bftau_{\gamma} + \sum_{m=1}^K \bfc_m\phi_m,
\end{equation}
where
\begin{equation}\label{eq:ph_rep}
\begin{aligned}
	&	\left(\Dcal_{\Gamma}^{\text{near}}\bftau_{\Gamma}\right)\left(\bfx\right):=\sum_{|n|\leq 1}\iG D(\bfx,\bfy+n\bfd)\bftau_{\Gamma}\left(\bfy\right)\ds_{\bfy}, \\
	&	\left(\Scal_{\gamma}^{\text{near}}\bftau_{\gamma}\right)\left(\bfx\right):=\sum_{|n|\leq 1}\ig S(\bfx,\bfy+n\bfd)\bftau_{\gamma}\left(\bfy\right)\ds_{\bfy}
\end{aligned}
\end{equation}
are sums over free-space kernels living on the walls and particle boundary in the central unit cell and its two near neighbors, and $\bfd$ is the the lattice vector i.e. $\bfd=\bfe_1$. The third term encodes the influence of the ``far'' periodic copies, where  $\phi_m(\bfx)=S(\bfx,\bfy_m)$ and the source locations $\{\bfy_m\}_{m=1}^K$ are chosen to be equispaced on a circle enclosing $\Omega$ \cite{marple:16}.

The unknown coefficients $\{\bfc_m\}_{m=1}^K$ are found by  enforcing the periodic inlet and outlet flow conditions at a set of collocation nodes. The resulting augmented linear system for the forward problem, for example, can be written in the following form in terms of the unknown density functions $\bftau_{\Gamma}$ and $\bftau_{\gamma}$ and the coefficients $\{\bfc_m\}$:
\begin{equation}\label{eq:ABCD_sys}
	\begin{bmatrix}
		A_{\Gamma,\Gamma} & A_{\Gamma,\gamma} & B_{\Gamma,\phi}\\
		A_{\gamma,\Gamma} & A_{\gamma,\gamma} & B_{\gamma,\phi}\\
		C_{\Gamma} & C_{\gamma} & D
	\end{bmatrix}\begin{bmatrix}\bftau_{\Gamma}\\ \bftau_{\gamma} \\ \bfc  \end{bmatrix} = 
	\begin{bmatrix}
		\bfu^D \\ \bfu^{\gamma} \\ \boldsymbol{0}
	\end{bmatrix}	
\end{equation}
The first row applies the slip condition on $\Gamma$ by taking the limiting value of $\bfu(\bfx)$, defined in \eqref{eq:vel_rep}, as $\bfx$ approaches $\Gamma$ from the interior. The second row uses the no slip condition on $\gamma$: $\lim_{\bfx\rightarrow \gamma}\bfu(\bfx)=\bfu^{\gamma}=\bdot{\bfx}{}\Gsup\shp\rho\bfr\cdot \bfx$. Then the centroid velocity $\bdot{\bfx}{}\Gsup$ and angular velocity $\rho$ can be solved for by applying extra force- and torque-free conditions. The third row applies the periodic boundary conditions on velocity and traction. The operators $A$, $B$, $C$, $D$ are correspondingly defined based on the representation formulas \eqref{eq:vel_rep} and \eqref{eq:ph_rep}.

The pointwise pressure and hydrodynamic traction for the ansatz \eqref{eq:vel_rep} are then given by
\begin{align}
	p = \Pcal_{\Gamma}^{D,\text{near}}\bftau_{\Gamma} + \Pcal_{\gamma}^{S,\text{near}}\bftau_{\gamma} + \sum_{m=1}^K \bfc_m P^S(\bfx,\bfy_m),\\
	 \bff = \Tcal_{\Gamma}^{D,\text{near}}\bftau_{\Gamma} + \Tcal_{\gamma}^{S,\text{near}}\bftau_{\gamma} + \sum_{m=1}^K \bfc_m T^S(\bfx,\bfy_m).
\end{align}
The operators in \eqref{eq:ABCD_sys} are discretized by splitting $\Gamma$ and $\gamma$ uniformly into $M_{\Gamma}$ and $M_{\gamma}$ disjoint panels respectively. In each panel, a $p$-th order Gauss-Legendre quadrature is employed to evaluate smooth integrals while a local panel-wise close evaluation scheme of \cite{wu2020solution} is employed to accurately handle corrections for the singularities of $S(\bfx,\bfy)$, $T^D(\bfx,\bfy)$ and $P^D(\bfx,\bfy)$. 
A forward Euler time-stepping scheme is used to evolve the particle position and the solution procedure outlined above is repeated at each time-step. 

In the case of the associated adjoint problems, the solution procedure remains the same but the right hand side of \eqref{eq:ABCD_sys} is modified according to the respective boundary conditions (e.g., (\ref{adjoint:PDE:particle}-d)). In addition, the particle velocities $\dot{\hat{\bfx}}^G$, $\hat{\rho}$ need to be computed by applying the total force and torque conditions, that is,  
given the traction vector $\hat{\bfh}$, the following condition is enforced:
\begin{equation}
	\igt \bff \ds_{\bfy} = \igt \hat{\bfh} \ds_{\bfy}\quad\text{and}\quad \bfe_3\cdot \igt \bfy\shtimes \bff\ds_{\bfy} = \bfe_3\cdot \igt \bfy\shtimes \hat{\bfh}\ds_{\bfy}.
\end{equation}

\section{Numerical Results}
\label{sc:results}

This section presents first validation tests of our boundary integral solvers and shape sensitivity formulas, then results on the shape optimization. In all numerical experiments, the following parameter values were used: $c=1$, $L = 2\pi$, $n=5$ (degree of the B-spline basis functions), and $M=7$. For the augmented Lagrangian optimization algorithm, we set $\zeta^{\star} = 0.01$, $\lambda^0 = (0,0)$, and $\sigma^0 = 10$.

\subsection{Validation of forward and adjoint PDE solvers}

To show the performance of periodic flow solver, we first solve a periodic Stokes flow problem with prescribed slip velocity, and test the convergence of the velocity field as we increase the number of quadrature points on $\Gamma$. We also show temporal convergence on forward and adjoint problem using forward Euler, where we set the axial distance $D=1$, $192$ quadrature points on $\Gamma^{\pm}$, and $60$ quadrature points on $\gamma$. Relative errors are shown below, where $F_0$ is the total force in adjoint problem on the particle at $t=0$, and $s_c$ is particle centroid at $D=1$.
    
\begin{figure}[!h]
   \includegraphics[height=.24\linewidth]{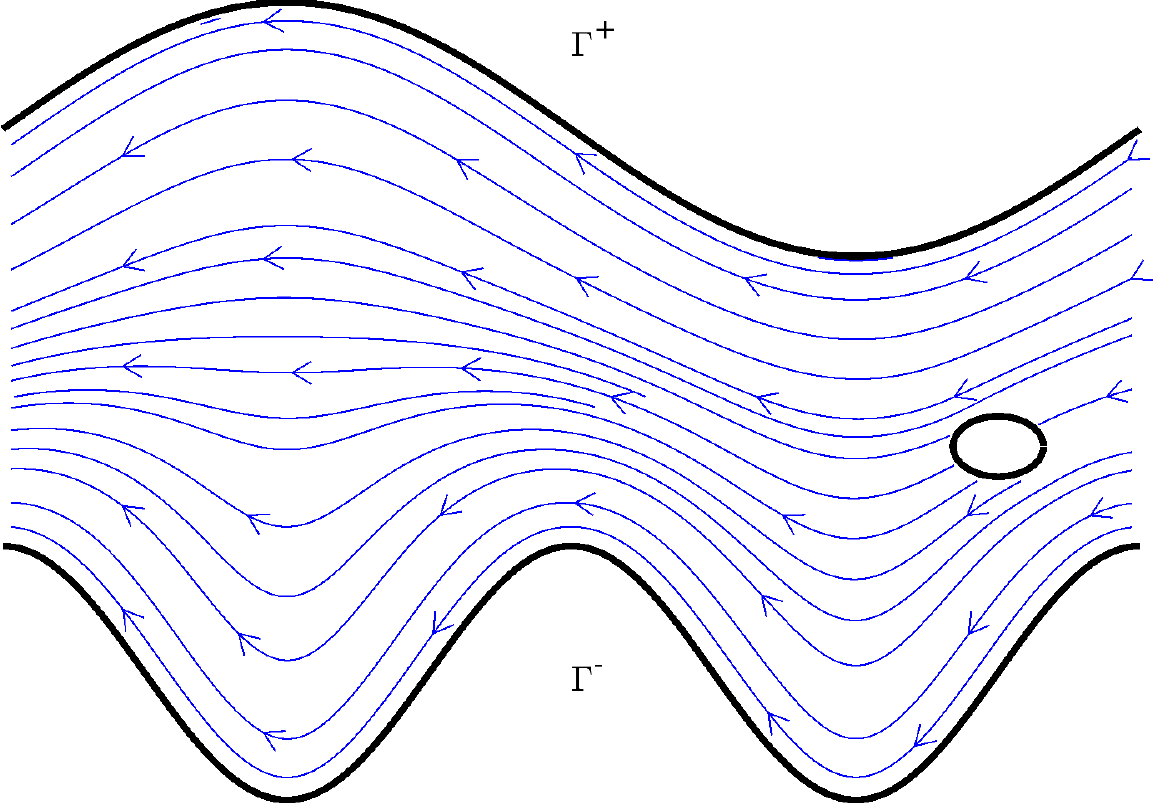}
   \includegraphics[height=.24\linewidth]{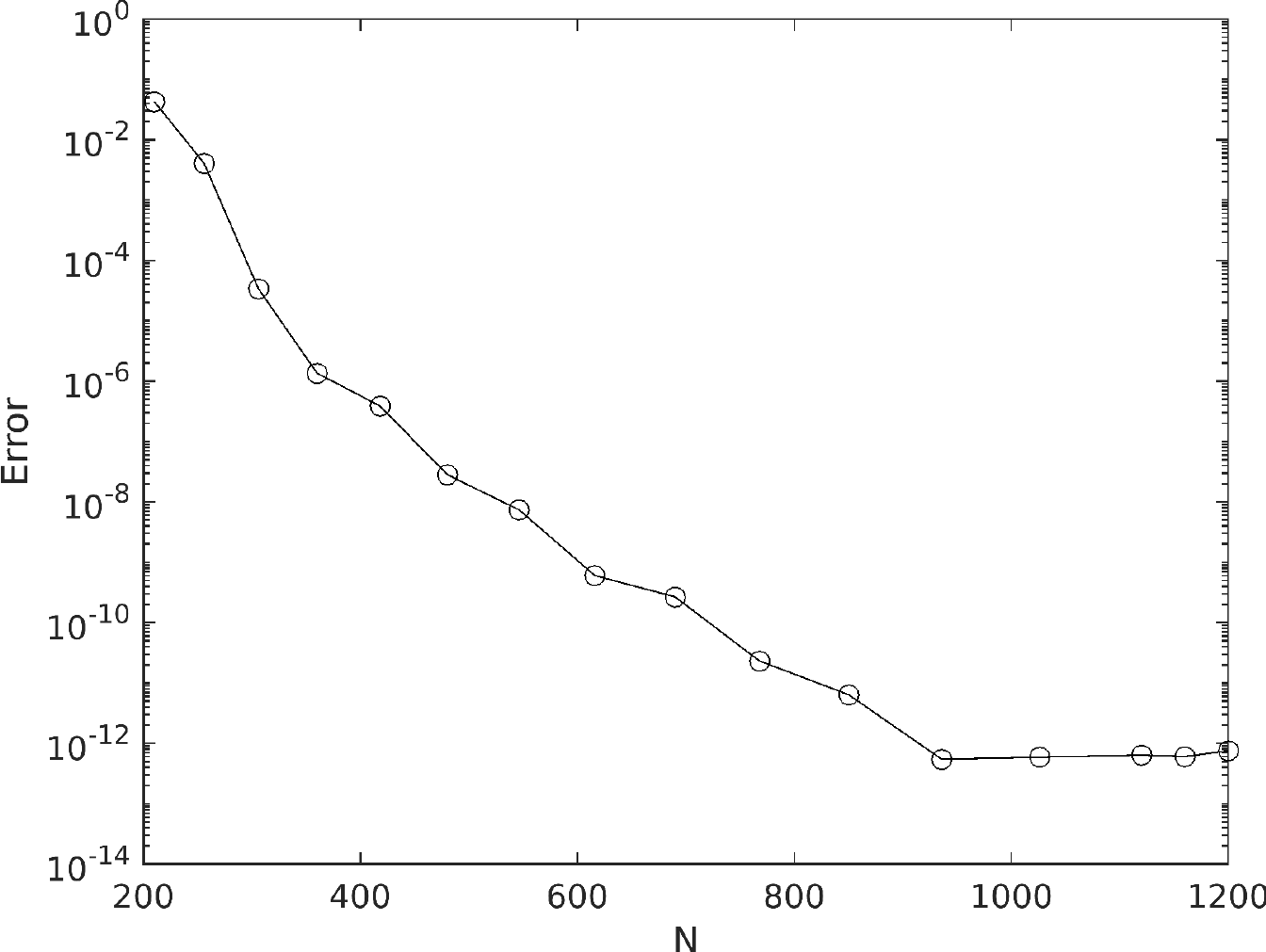}
   \includegraphics[height=.24\linewidth]{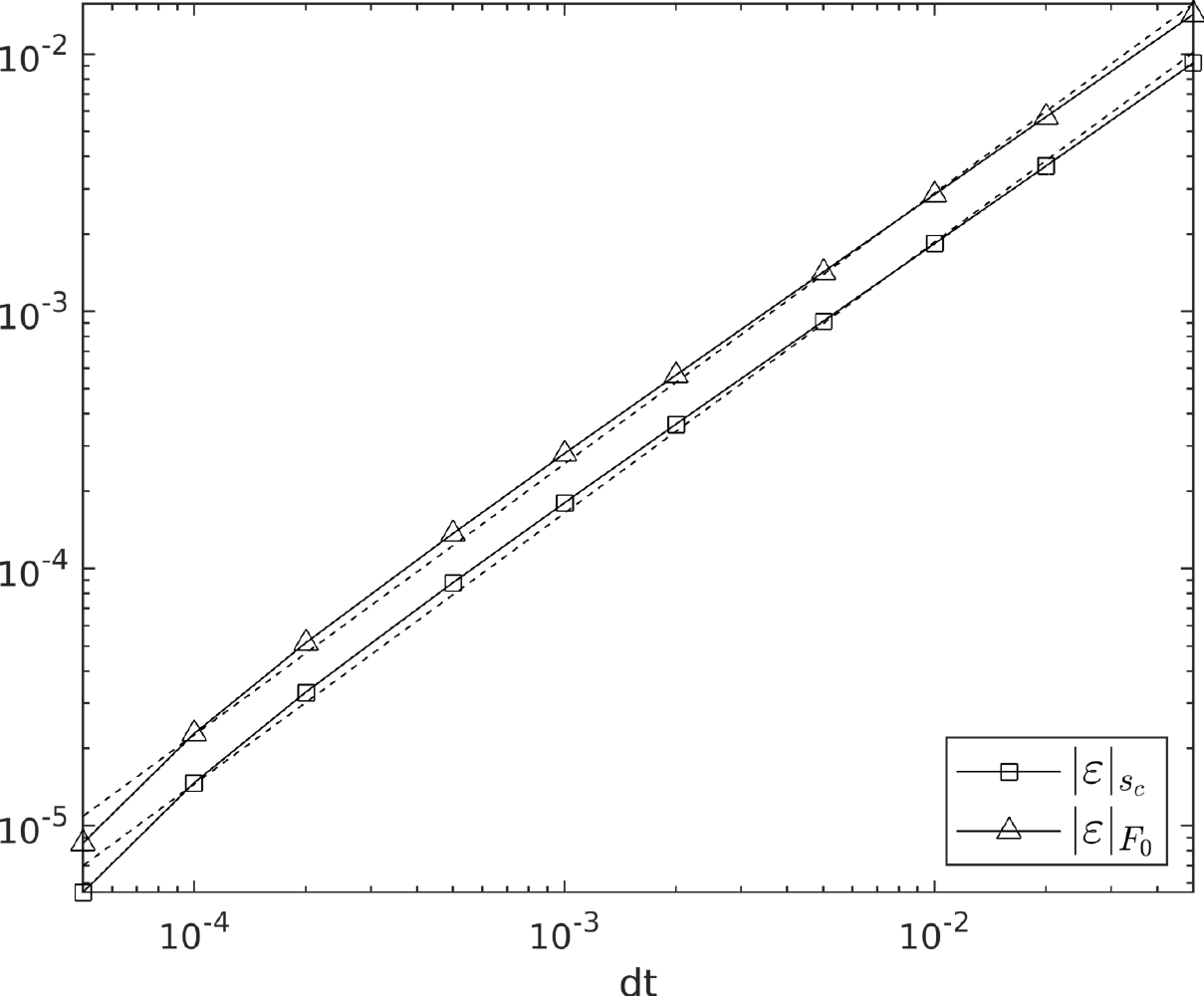}
    \caption[]{\em Validation of the numerical solver. (Left) Streamlines of a periodic Stokes flow induced by prescribed slip on the walls, obtained using our boundary integral solver. (Middle) Plot of self convergence as a function of the spatial resolution $N=pM_{\Gamma}$ used in the forward solver. (Right) Temporal validation. Reference values are computed using $dt=2\times 10^{-5}$.} \label{fig:bie_validation}
\end{figure}


\subsection{Validation of analytical shape sensitivity formulas}\label{sec:validation}

We consider a sinusoidal wall shape and a circular particle shape (Fig.~\ref{fig:validate}a) and compare the shape sensitivities obtained by the finite difference approach and the analytical sensitivity formulas derived in Section \ref{sc:sensitivities}. We use the central difference scheme to approximate the shape derivative:  
\begin{equation}\label{eq:Jprime}
   J'_{\text{FD}} =  J'(\bfx^{\pm};\bfth^{\pm}) = \frac{1}{2\eta} \Lsqb J\lpar \bfx^{\pm}\left(t;\bfxi(\bfps)\right) + 
    \eta \bfth^{\pm} \rpar - J\lpar \bfx^{\pm}\left(t;\bfxi(\bfps)\right) -\eta \bfth^{\pm} \rpar \Rsqb
\end{equation}
with step size $\eta = 10^{-4}$. Here, $J$ is either the energy dissipation functional $J\Wsub$ or the net motion $D$ in the wave frame. Substituting~\eqref{eq:transformation} in~\eqref{eq:Jprime}, we get the following simplified expression, 
\begin{equation}\label{eq:Jprime2}
    J'_{\text{FD}} = J'(\bfx^{\pm};\bfth^{\pm}) = \frac{1}{2\eta} \Lsqb J\lpar \bfx^{\pm}\left(t;\bfxi(\bfps\shp\eta \bfdel)\right) \rpar - J\lpar \bfx^{\pm}\left(t;\bfxi(\bfps\shm\eta \bfdel)\right) \rpar \Rsqb ,
\end{equation}
where $\bfdel$ is a standard basis vector. Depending on the index of the nonzero element in $\bfdel$, there are $(4M\shm2)$ possible shape perturbations. These serve as the basis of any arbitrarily smooth perturbation of the wall shape. A comparison of the shape sensitivities evaluated by the two methods is shown in Fig.~\ref{fig:validate}b, which validates the analytical shape sensitivity formulas using finite difference approach as reference.

\begin{figure}[!h]
    \centering
    \includegraphics[width = \textwidth]{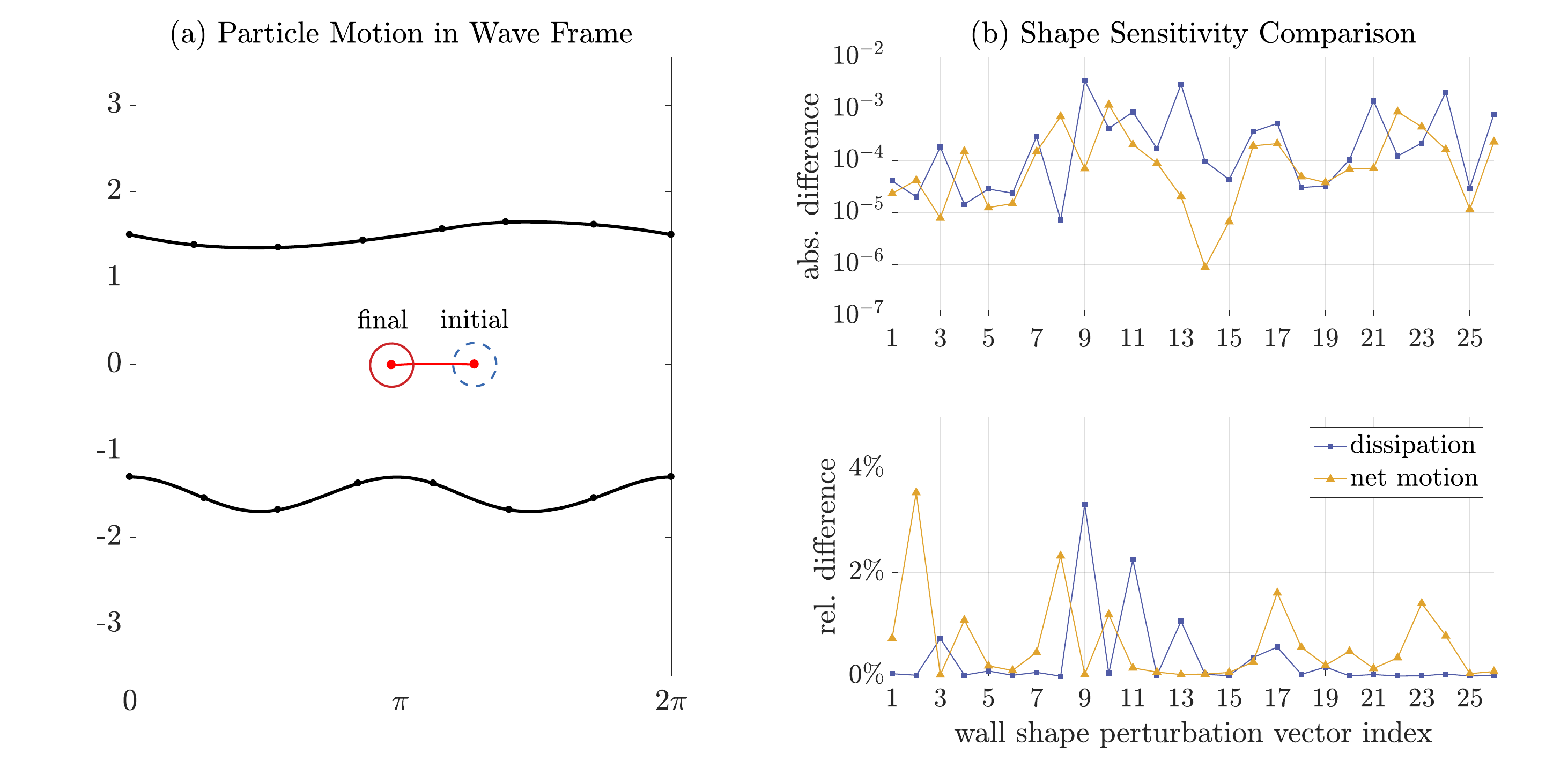}
    \caption{(a) The wall shape and the motion of the particle (centroid) in wave frame. The dots on the wall display the control points. (b) Comparison of the shape sensitivities of dissipation (blue) and net motion in wave frame (red), using analytical and finite difference approaches for the example. The wall shape perturbation vector index is the index of the nonzero element in the perturbation vector $\bfdel$. The absolute difference is $|J'_{\text{analytic}} - J'_{\text{FD}}|$ and the relative difference is $|J'_{\text{analytic}} -J'_{\text{FD}}|/|J'_{\text{FD}}|$.
    }
    \label{fig:validate}
\end{figure}

\subsection{Optimization experiments}\label{sec:optresults}
Here, we present results on the numerical optimization of peristaltic pumps carrying a rigid particle. Figure \ref{fig:motion} shows the optimal wall shapes obtained by our algorithm for different net particle motions with
the same volume of fluid region. As expected, the optimal value of dissipation increases for faster net particle velocity in the fixed frame. In the extreme case where the net velocity of the particle is zero in the fixed frame, as expected, the optimal shape is a flat channel with no dissipation. On the other hand, when the particle moves at the same speed as the peristaltic wave, the centroid of the particle remains fixed in the wave frame.\enlargethispage*{1ex}

\begin{figure}[b]
    \centering
    \includegraphics[width = 0.95\textwidth]{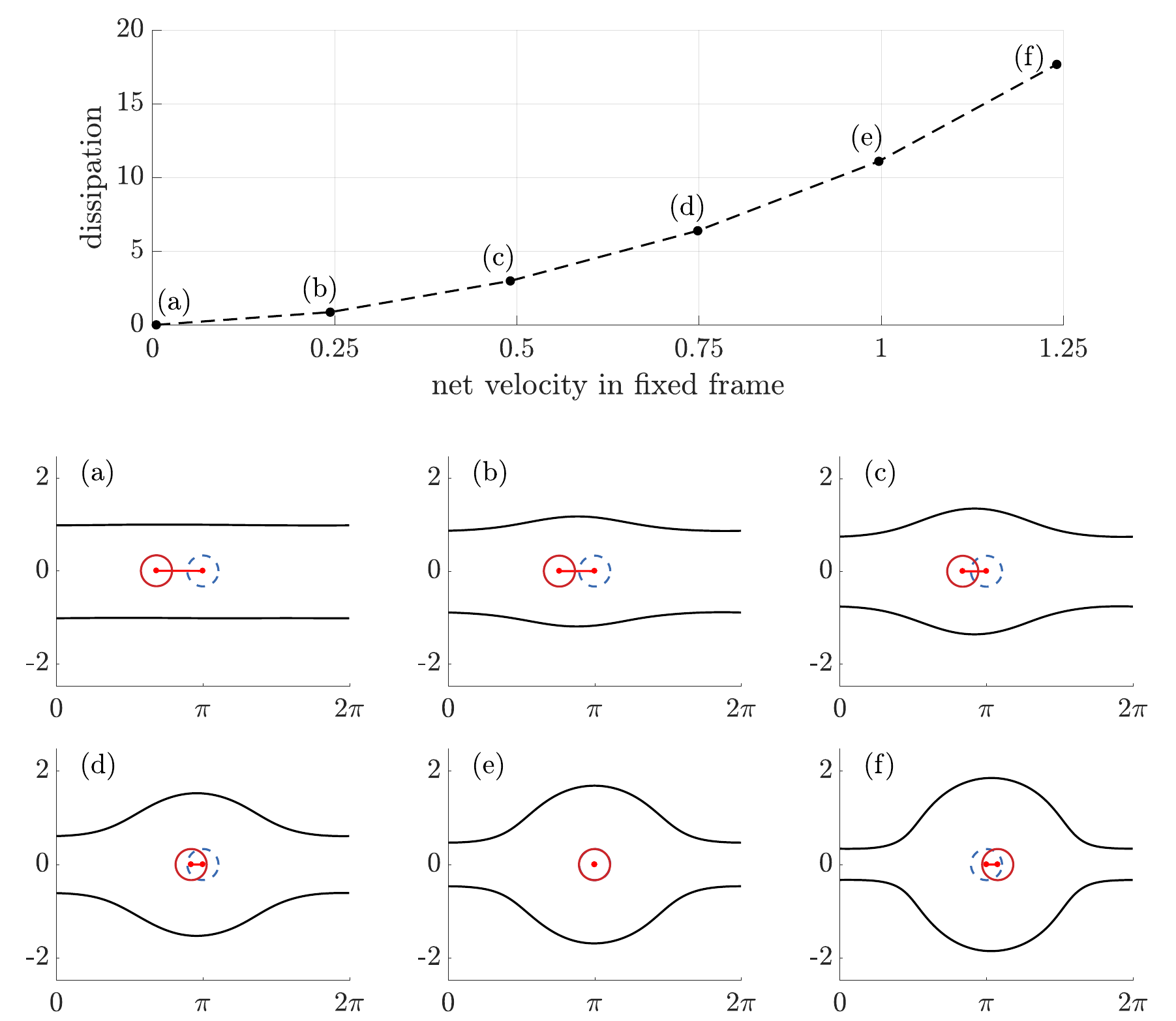}
    \caption{Optimal channel wall shapes (in wave frame) for varying net particle velocity. $V_0=12.26$ and $T = 1$ for all experiments. The particle motions are shown by the initial location (dashed blue), the final location (solid red) and the trajectory of the centroid. Fig. (e) shows a scenario that the particle moves at the same speed as the peristalsis pumping wave speed.}
    \label{fig:motion}
\end{figure}

In Fig.~\ref{fig:iterA}, we plot the wall shapes as the optimization progresses for the case of Fig.~\ref{fig:motion}(e). They evolve from an arbitrary initial channel wall shape to reach a configuration achieving the target volume $|\OO|=V_0$ and net particle motion $D=0$ (i.e., a unit net velocity in the fixed frame). The values of the augmented Lagrangian objective $\Lcal\Asub$, dissipation $J\Wsub$, volume of fluid region $V$ and net particle motion in the wave frame $D$ are shown in Fig.~\ref{fig:iterB}.

\begin{figure}[t]
    \centering
    \includegraphics[width = \textwidth]{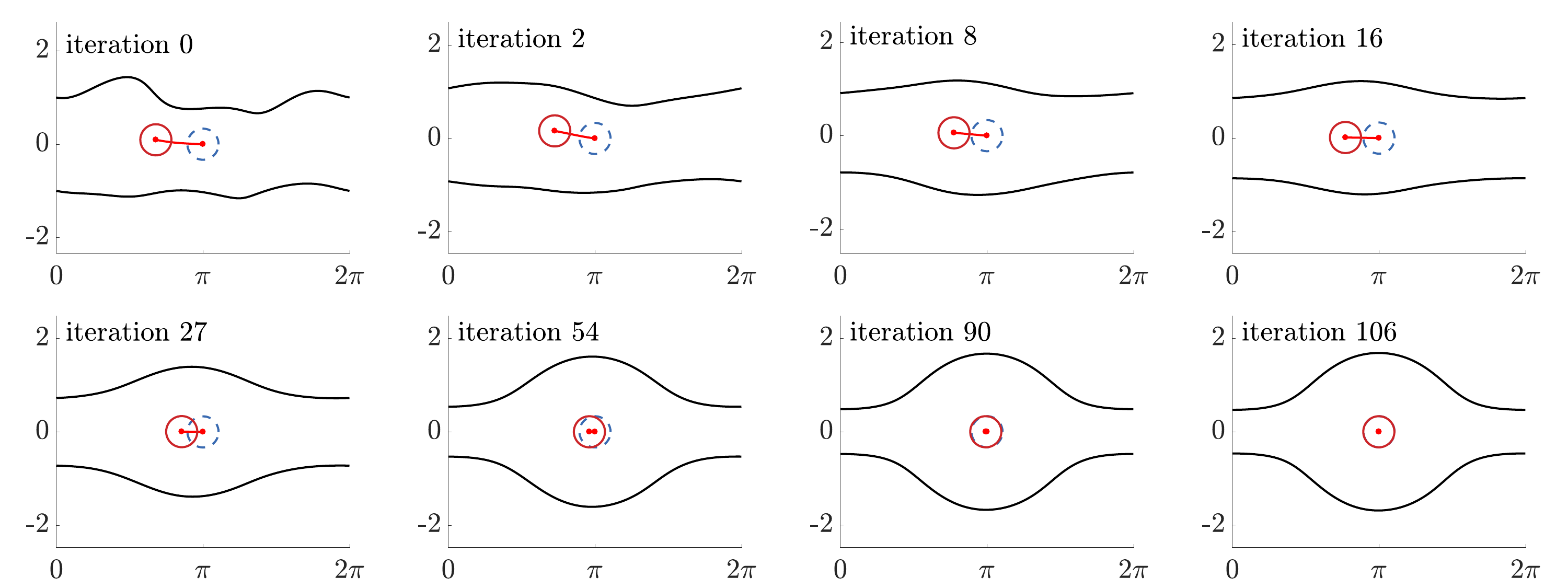}
    \caption{Optimization process of Fig.~\ref{fig:motion}(e) starting from an arbitrary shape. The total number of BFGS iterations is 106 and the total number of evaluations of the problem is 124.}
    \label{fig:iterA}

\bigskip

    \centering
    \includegraphics[width = 0.95\textwidth]{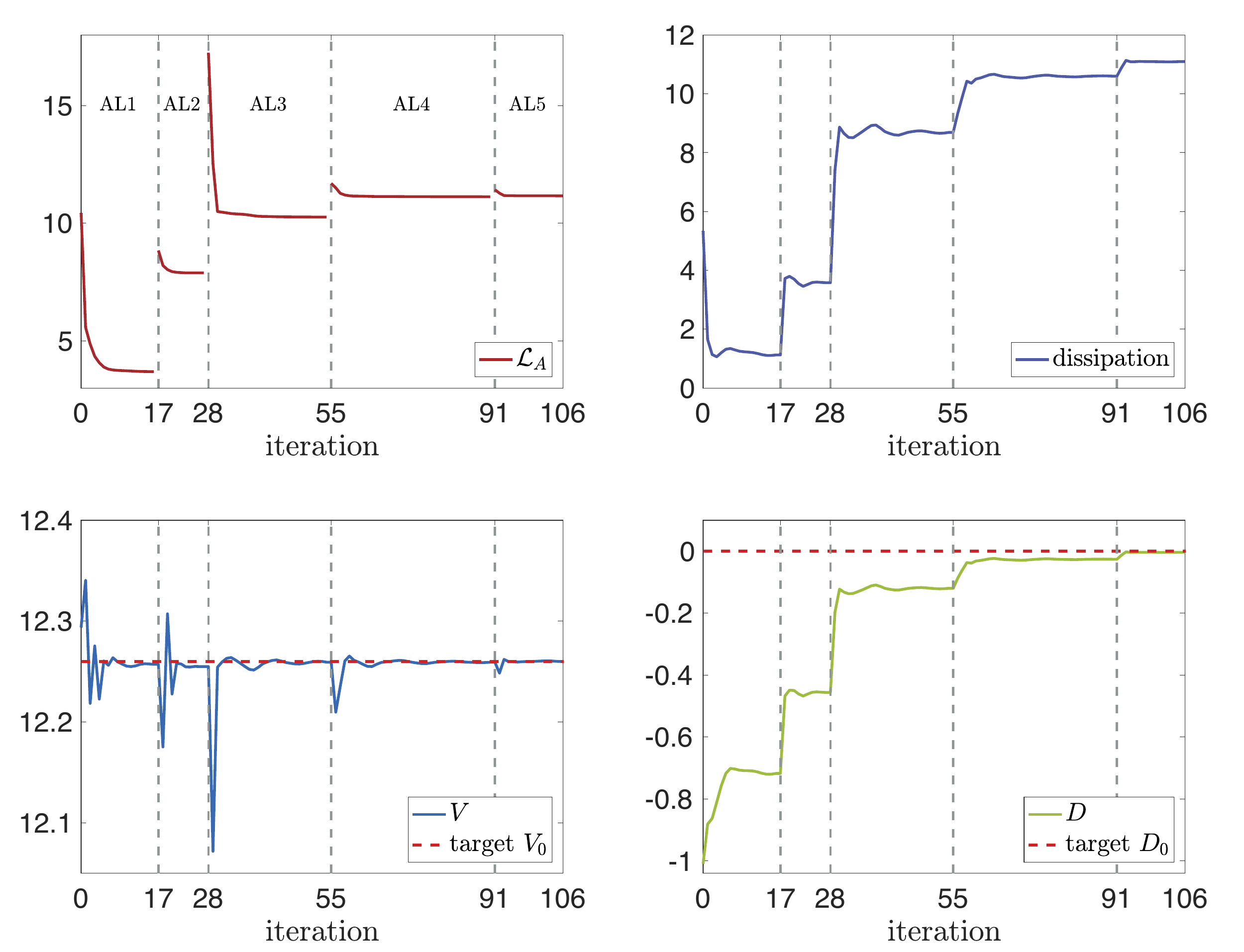}
    \caption{Quantities for the optimization process of Fig.~\ref{fig:motion}(e). The vertical dashed lines label the restart of augmented Lagrangian (AL) after the local minimum is achieved and the penalty parameters or the Lagrangian multipliers is therefore updated.}
    \label{fig:iterB}
\end{figure}

\begin{figure}[t]
    \centering
    \includegraphics[width = \textwidth]{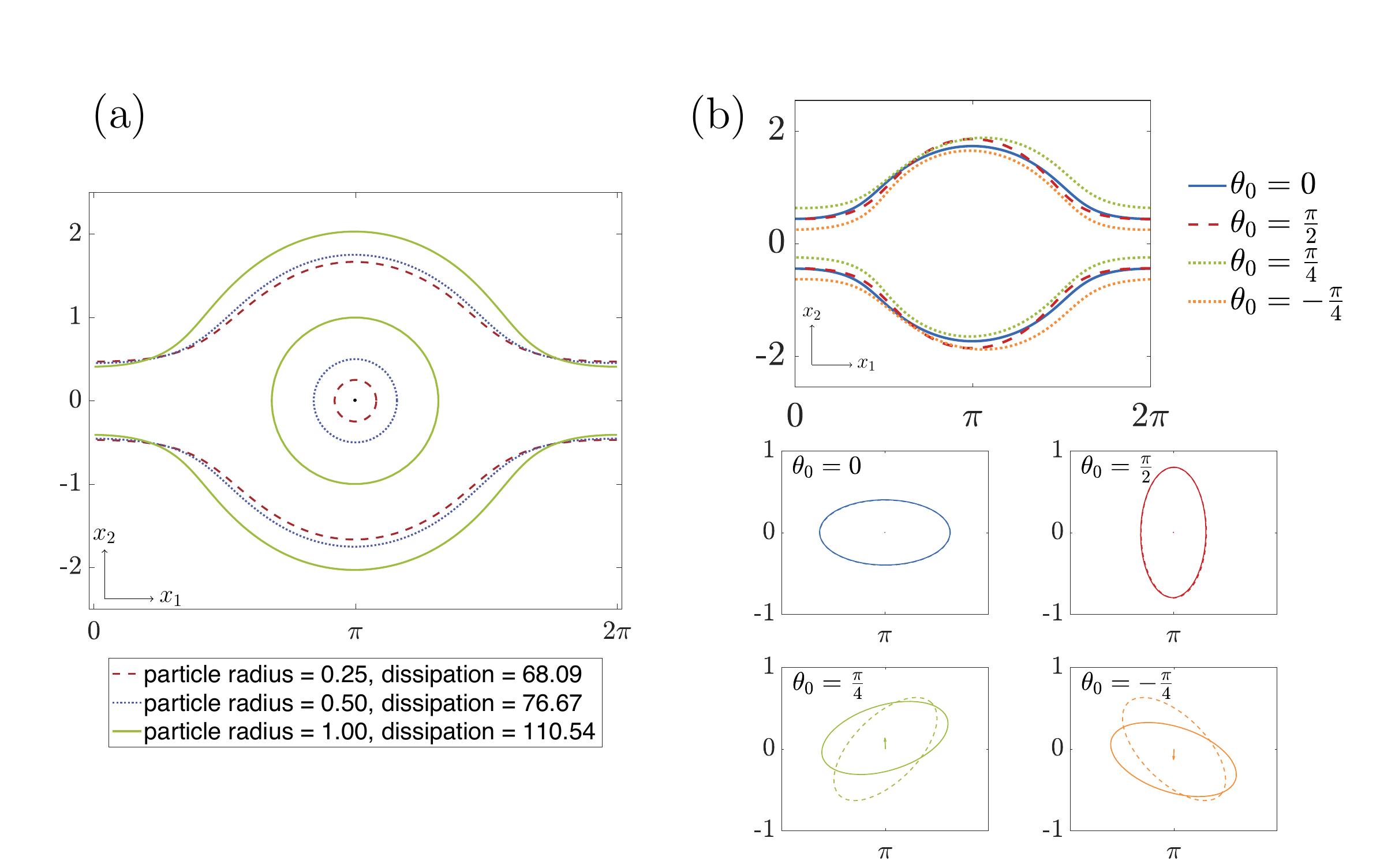}\vspace{-1ex}
    \caption{Optimal wall shapes and particle motions (in wave frame) for varying particle configurations. Here we use $V_0=12.26$, $D_0=0$, $T = 6.3$ and initial particle centroid at $(\pi,0)$ for all cases. (a) For the circle particle, the particle motion in the wave frame is static because in the fixed frame it is moving at the same speed as the wave speed and no vertical translation is observed. (b) For the ellipse particle, the initial tilting angle $\theta_0$ affects the optimal wall shape. The motion of the ellipse particle is shown by initial position (dashed) to final position (solid) corresponding to the color for each $\theta_0$.}
    \label{fig:shape}
\end{figure}

\begin{figure}[t]
    \centering
    \includegraphics[width = \textwidth]{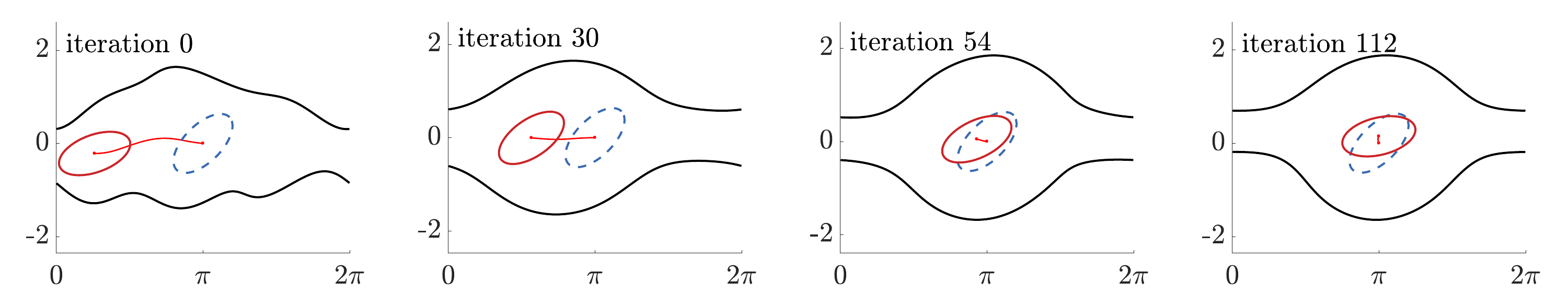}\vspace{-1ex}
    \caption{Optimization process of the ellipse particle with $\theta_0 = \pi/4$ in Fig.~\ref{fig:shape}.  The total number of BFGS iterations is 112 and the total number of evaluations of the problem is 153.}
    \label{fig:iter_ellipse}
\end{figure}

Next, to illustrate the particle effect on the optimal wall shapes, we ran experiments on different shapes and sizes of particles and show the results in Fig.~\ref{fig:shape}. Specifically, we consider circular particles of different size and elliptical particles at different orientations. The initial location of the particle centroid is set to $(\pi,0)$ in all cases. A common feature we find across all the shapes and sizes of particles is that for minimum dissipation, they are carried at the center of fluid domain in the wall frame. Fig.~\ref{fig:iter_ellipse} displays the progression of the optimization starting from an arbitrary pipe shape, where this phenomenon can be observed clearly.\enlargethispage*{1ex}

The results of the numerical experiments show that the optimal pumping wall shapes form an enclosing bolus around the rigid particle near the center line of the channel. Particularly, for a fixed-sized circle particle, a larger target net velocity leads to a bolus of larger size, as seen in Fig.~\ref{fig:motion}. For a fixed net velocity, a larger particle leads to a bigger bolus, see Fig.~\ref{fig:shape}. For the case of an elliptical particle, the initial angle $\theta_0$ between its major axis and the center line affects the symmetry of the optimal channel geometry. For example, setting $\theta_0=\pi/4$, the converged wall shape forms an asymmetric bolus with a more-deformed lower wall, see Fig.~\ref{fig:iter_ellipse}. Comparing the final wall geometry reached for different initial orientations $\theta_0$ of the elliptical particle, shown in Fig.~\ref{fig:shape}(b), the bolus is symmetric about $x_1=0$ and $x_2=\pi$ for $\theta_0 = 0$ or $\pi/2$, with a slightly larger vertical amplitude in the latter case, while the boluses found for $\theta_0 = \pm \pi/4$ are slightly asymmetric. The center line of the channel is shifted upwards for $\theta_0 = \pi/4$ and downwards for $\theta_0 = -\pi/4$. Additionally, for the case $\theta_0 = \pi/4$, the lower wall is still nearly horizontally symmetric about $x_2 = \pi$ but the upper wall is not, its largest amplitude shifting to the right. The walls for $\theta_0 = -\pi/4$ show the opposite trend: the upper wall is nearly symmetric about $x_2 = \pi$ while the lower wall is not symmetric with its largest amplitude shifted to the right.
The initially-tilted elliptical particle ($0\shl\theta_0\shl\pi/2$) undergoes body rotation and vertical translation. For example, the particle motions for $\theta_0 = \pm \pi/4$ experience opposite rotations (respectively clockwise and counterclockwise) and translations (respectively upwards and downwards).

\section{Conclusions} \label{sc:conclusions}
We presented a gradient-based optimization approach for finding the optimal shapes of peristaltic pumps for transporting rigid particles in Stokes flow. While we considered the power loss functional and associated constraints, the procedure for deriving shape sensitivities generalizes to other related objective functions and constraints. An important contribution of this work is an adjoint formulation that, in conjunction with a boundary integral formulation, significantly reduces the computational burden of evaluating shape derivatives in the case of particulate flows. 

Although we restricted our attention to peristaltic pumps, the computational framework developed here is applicable to a wide range of design and optimization problems in interfacial fluid mechanics. For example, we recently applied similar techniques to optimize the swimming action of axisymmetric microswimmers \cite{guo2021optimal, guo2021optimal2}.  Extensions to time-dependent problems such as deformable microswimmers (e.g., cells driven by membrane deformations \cite{farutin2013amoeboid}) or active flows in complex geometries \cite{bechinger2016active} can benefit from the adjoint formulation developed here.

\subsection*{Acknowledgements}
RL, SV and HZ acknowledge support from NSF under
grants DMS-1454010 and DMS-2012424. The work of SV was also supported by the Flatiron Institute (USA), a division of Simons Foundation, and by the  Fondation Math\'{e}matique Jacques Hadamard (France).

\bibliographystyle{siamplain}      
\bibliography{ref}   

\appendix
\section{Proofs}

\subsection{Differential operators on curved boundaries}\label{grad:curv}

In preparation for some of the proofs to follow, we list useful formulas and notations regarding differential operators evaluated on curved boundaries of a fluid domain.
Let points $\bfx$ in a tubular neighborhood $V$ of $\G$ be represented as\enlargethispage*{1ex}
\begin{equation}
  \bfx = \bfx(s) + z\bfn(s)
\end{equation}
in terms of curvilinear coordinates $(s,z)$, and let $\bfv(\bfx)=v_s(s,z)\bftau(s)\shp v_n(s,z)\bfn(s)$ denote a generic vector field in $V$. Then, at any point $\bfx\sheq\bfx(s)$ of $\G$, we have
\begin{equation}
  \bfna\bfv = \del{s}\bfv \tensor\bftau + \del{n} v_s\bftau\tensor\bfn + \del{n} v_n\bfn\tensor\bfn, \qquad
  \Div\bfv  = \bftau\sip\del{s}\bfv + \del{n} v_n \label{grad:div:expr}.
\end{equation}
Assuming incompressibility, the condition $\Div\bfv\sheq0$ can be used to eliminate $\del{n} v_n$, yielding the following expressions of $\bfna\bfv$ and $2\bfD[\bfv]=\bfna\bfv+\bfna\bfv\Tsup$:
\begin{align}
  \bfna\bfv
 &= \del{s}\bfv \tensor\bftau + \del{n} v_s\bftau\tensor\bfn
  - \lpar \bftau\sip\del{s}\bfv \rpar\bfn\tensor\bfn \\
 &= \lpar \bftau\sip\del{s}\bfv \rpar \lpar \bftau\tensor\bftau \shm \bfn\tensor\bfn \rpar + \lpar \bfn\sip\del{s}\bfv \rpar\bfn\tensor\bftau + \del{n} v_s\bftau\tensor\bfn, \label{D[u]:inc} \\
  2\bfD[\bfv]
 &= 2\lpar \bftau\sip\del{s}\bfv \rpar \lpar \bftau\tensor\bftau \shm \bfn\tensor\bfn \rpar
 + \lpar \bfn\sip\del{s}\bfv \shp \del{n} v_s \rpar \lpar \bfn\tensor\bftau \shp \bftau\tensor\bfn \rpar
\end{align}
Next, we evaluate the stress vector $\bff=-p\bfn+2\bfD[\bfv]\sip\bfn$, to obtain
\begin{equation}
  \bff = f_s\bftau + f_n\bfn \qquad\text{with} \quad
  f_s = \bfn\sip\del{s}\bfv \shp \del{n} v_s , \quad f_n = -p -2(\bftau\sip\del{s}\bfv).
\end{equation}
In particular, we therefore have $\del{n} v_s = f_s\shm \bfn\sip\del{s}\bfv$, allowing (by eliminating the remaining normal derivative therein) to express $\bfna\bfv$ and $\bfD[\bfv]$ in terms of quantities defined on the boundary:
\begin{equation}
\begin{aligned}
  \bfna\bfv
 &= \lpar \bftau\sip\del{s}\bfv \rpar \lpar \bftau\tensor\bftau \shm \bfn\tensor\bfn \rpar + \lpar \bfn\sip\del{s}\bfv \rpar \lpar \bfn\tensor\bftau \shm \bftau\tensor\bfn \rpar  + f_s\bftau\tensor\bfn, \\
  2\bfD[\bfv]
 &= 2\lpar \bftau\sip\del{s}\bfv \rpar \lpar \bftau\tensor\bftau \shm \bfn\tensor\bfn \rpar + f_s \lpar \bfn\tensor\bftau \shp \bftau\tensor\bfn \rpar
\end{aligned}\label{grad:D}
\end{equation}

\subsubsection*{Velocity with rigid-body boundary traces} In this case, we consider vector fields $\bfu$ satisfying $\bfu=\bfu_0 \shp\varrho\bfr\sip\bfx$ on the particle boundary $\gamma$, see~\eqref{rigid:velocity}. This implies
\begin{equation}
  \del{s}\bfu = \varrho\bfr\sip\del{s}\bfx = \varrho\bfr\sip\bftau = \varrho\bfn \qquad \text{on $\gamma$}
\end{equation}
(since $\bfr=\bfe_2\tensor\bfe_1\shm\bfe_1\tensor\bfe_2=\bfn\tensor\bftau\shm\bftau\tensor\bfn$ with the orientation convention of Fig.~\ref{geom:2D}), so that~\eqref{D[u]:inc} implies
\begin{equation}
  \bfna\bfu = \varrho\bfn\tensor\bftau + \del{n}\bfu_s\bftau\tensor\bfn, \quad
  2\bfD[\bfu] = \lpar \del{n}\bfu_s \shp \varrho \rpar\,\lpar \bfn\tensor\bftau \shp \bftau\tensor\bfn \rpar,
  \quad \DivS\bfu=0. \label{Dv:aux:gamma}
\end{equation}
on $\gamma$. We now evaluate the stress vector $\bfh:=\bfsig[\bfu,p]\sip\bfn=-p\bfn+2\bfD[\bfu]\sip\bfn$, to obtain
\begin{equation}
  \bfh = -p\bfn + \lpar \del{n}\bfu_s \shp \varrho \rpar\bftau = -p\bfn + h_s\bftau \qquad \text{and hence \ }
  \del{n}\bfu_s = h_s\shm\varrho.
\end{equation}
Using the above formula for $\del{n}\bfv_s$ in~\eqref{Dv:aux:gamma}, we finally obtain
\begin{equation}
  \text{(a) \ }\bfna\bfu = \varrho\bfr + h_s\bftau\tensor\bfn, \qquad
  \text{(b) \ }2\bfD[\bfu] = h_s\lpar \bfn\tensor\bftau \shp \bftau\tensor\bfn \rpar. \label{Dv:gamma}
\end{equation}

\subsection{Proof of Lemma~\ref{DivS:theta}}\label{DivS:theta:proof} The lemma follows directly from using~(\ref{Dv:gamma}a) in~\eqref{DivS:def}.

\subsection{Proof of formula (\ref{lm4})}\label{app}

We use the Frenet formulas~\eqref{frenet} and associated conventions. To evaluate $\dd{\bfu}{}\Dsup$, we let $\G$ depend on the fictitious time $\eta$, setting
\begin{equation}
\G_{\eta}\ni\bfx_{\eta}(s) = \bfx(s)+\eta\bfth(s) \qquad(0\shleq s\shleq \ell), \label{wall:shape:eta:2D}
\end{equation}
(where $\G$ stands for $\G^+$ or $\G^-$, and likewise for $\ell$) and seek the relevant derivatives w.r.t. $\eta$ at $\eta=0$. Note that for $\eta\not=0$, $s$ is no longer the arclength coordinate along $\G_{\eta}$, and $\del{s}\bfx_{\eta}$ is no longer of unit norm; moreover, the length of $\G_{\eta}$ depends on $\eta$. The wall velocity $\bfU=\ell\bftau$ for varying $\eta$ is then given by
\begin{equation}
  \bfU_{\eta}(s) = (\ell_{\eta}/g_{\eta})\, \del{s}\bfx_{\eta}  \qquad(0\shleq s\shleq \ell), \label{U:def:eta:2D}
\end{equation}
having set $g_{\eta}=|\del{s}\bfx_{\eta}|$ (note that $g_0=1$). Our task is to evaluate $\text{d}/\text{d}\eta\,\bfU_{\eta}(s)$ at $\eta=0$. We begin by observing that the derivative of $g$ is (since $\del{s}\bfx_{\eta}=\bftau$ and $g=1$ for $\eta=0$)
\begin{equation}
  \del{\eta}g
   = (\del{s}\bfx_{\eta}\sip\del{\eta s}\bfx_{\eta})/g = \bftau\sip\del{s}\bfth
   = \del{s}\theta_s - \kappa\theta_n
\end{equation}
and the length $\ell_{\eta}$ of $\G_{\eta}$ and its derivative $\dd{\ell}$ are given (noting that $s$ spans the fixed interval $[0,\ell]$ for all curves $\G_{\eta}$) by
\begin{equation}
  \text{(i) \ } \ell_{\eta} = \int_{0}^{\ell} g_{\eta} \ds, \qquad
  \text{(ii) \ } \dd{\ell} = \int_{0}^{\ell} (\del{s}\theta_s - \kappa\theta_n) \ds
 = -\int_{0}^{\ell} \kappa\theta_n \ds.
\label{dd:ell:2D}
\end{equation}
The last equality in (ii), which results from the assumed periodicity of $\bfth$, is item (b) of~\eqref{lm4}. Then, using the above formulas in~\eqref{U:def:eta:2D} establishes item (a) of~\eqref{lm4}, as we find
\begin{equation}
 \dd{\bfu}{}\Dsup = \del{\eta}\bfU_{\eta}(s)\big|_{\eta=0}
 = \lsqb \dd{\ell} - \ell(\del{s}\theta_s - \kappa\theta_n) \rsqb\bftau + \ell\del{s}\bfth
 = \dd{\ell}\bftau + \ell(\del{s}\theta_n + \kappa\theta_s)\bfn. \label{dd:uD:2D}
\end{equation}

\subsection{Proof of Lemma~\ref{lm5}}\label{lm5:proof}

Firstly, it is straightforward (e.g. using component notation) to show that $\Div\bfE\Tsup\lpar(\bfu,p),(\bfu,p)\rpar = \bfze$, i.e. $\del{j}E_{ji} = 0$ $(i\sheq1,2)$ holds for any $(\bfu,p)$ satisfying $\Div\bfu\sheq0$ and $-\Delta\bfu\shp\bfna p\sheq\bfze$. For any vector field $\bfzet\in C^{1,\infty}_0(\OO_{\text{all}})$, we consequently have
\begin{align}
  \bfE\lpar(\bfu,p),(\bfu,p)\rpar\dip\bfna\Tsup\bfzet
 &= \Div\lsqb\bfE\lpar(\bfu,p),(\bfu,p)\rpar\sip\bfzet\rsqb - \lsqb\Div\bfE\Tsup\lpar(\bfu,p),(\bfu,p)\rpar\rsqb\sip\bfzet \\
 &= \Div\lsqb\bfE\lpar(\bfu,p),(\bfu,p)\rpar\sip\bfzet\rsqb. \label{E:div}
\end{align}
Then, observing that $\lpar(\bfu,p),(\bfv,q)\rpar\mapsto\bfE\lpar(\bfu,p),(\bfv,q)\rpar$ defines a symmetric bilinear form, we invoke the polarization identity and obtain
\begin{align}
  \bfE\lpar(\bfu,p),(\bfuh,\hatp)\rpar\dip\bfna\Tsup\bfzet
 &= \pinv{4} \lsqb \bfE\lpar(\bfu\shp\bfuh,p\shp\hatp),\, (\bfu\shp\bfuh,p\shp\hatp) \rpar
 + \bfE\lpar (\bfu\shm\bfuh,p\shm\hatp),\, (\bfu\shm\bfuh,p\shm\hatp)\rpar\rsqb \dip\bfna\bfzet \\
 &= \pinv{4} \Div\lsqb \bfE\lpar(\bfu\shp\bfuh,p\shp\hatp),\, (\bfu\shp\bfuh,p\shp\hatp) \rpar\sip\bfzet
 + \bfE\lpar (\bfu\shm\bfuh,p\shm\hatp),\, (\bfu\shm\bfuh,p\shm\hatp)\rpar\sip\bfzet \rsqb \\
 &= \Div\lpar\bfE\lpar(\bfu,p),(\bfuh,\hatp)\rpar\sip\bfzet\rpar.
\end{align}
Hence, applying the first Green identity (divergence theorem) yields
\begin{equation}
  \lbra\bfE\lpar(\bfu,p),(\bfuh,\hatp)\rpar,\bfna\Tsup\!\bfzet\rbra_{\OO}
 = \int_{\G\cup\gamma\cup\G_0\cup\G_L} \bfn\sip\bfE\lpar(\bfu,p),(\bfuh,\hatp)\rpar\sip\bfzet \ds.
\end{equation}
Finally, condition (ii) in~\eqref{Theta:def} and the assumed periodicity conditions at the end sections for the velocity fields (which imply the same periodicity for $\bfna\bfu$ and $\bfna\bfuh$ by the known interior regularity of $\bfu,\bfuh$ in the whole channel), for $p$ (but not necessarily for $\hatp$) as well as for $\bfzet$ give
\begin{equation}
  \int_{\G_0\cup\G_L} \bfn\sip\bfE\lpar(\bfu,p),(\bfuh,\hatp)\rpar\sip\bfzet \ds
 = \int_{\G_L} \Delta\hatp (\del{2} u_1) \zeta_2 \ds,
\end{equation}
which completes the proof of the claimed integral identity.

Then, if $\bfu$ and $\bfuh$ are rigid-body velocity fields on $\gamma$, \eqref{Dv:gamma} provides $2\bfD[\bfu]\dip\bfD[\bfuh] = h_s\hath_s$, $\bfh\sip\bfna\bfuh = \ooh\bfh\sip\bfr + h_s\hath_s\bfn$ and $\bfhh\sip\bfna\bfu = \varrho\bfhh\sip\bfr + h_s\hath_s\bfn$, from which formula~\eqref{E:gamma} readily follows.

\subsection{Proof of equation~\eqref{forward:wall}}
\label{proof:lemma5}

When applied to the solution of the forward problem~\eqref{forward:weak}, which satisfies $\bfu=(c\ell/L)\bftau$ on $\G$, formulas~\eqref{grad:D} yield
\begin{equation}
  \bfna\bfu
 = \kappa\ell \lpar \bfn\tensor\bftau\shm\bftau\tensor\bfn \rpar + f_s\bftau\tensor\bfn, \quad
  2\bfD[\bfu] = f_s \lpar \bfn\tensor\bftau \shp \bftau\tensor\bfn \rpar, \quad
  f_n = -p \qquad\text{on $\G$} \label{Du:expr}
\end{equation}
(in particular, the viscous part of $\bff$ is tangential to $\G$). Moreover, we also readily obtain
\begin{equation}
  2\bfD[\bfu]\dip\bfD[\bfuh] = f_s\fhat_s, \qquad
  \bff\sip\bfna\bfuh\sip\bfth
 = \lpar\bff\sip\del{s}\bfuh\rpar\theta_s
 + \Lpar \lpar p\bftau\shm f_s\bfn \rpar\sip\del{s}\bfuh + f_s\fhat_s \Rpar \theta_n, \label{DuDu}
\end{equation}
and using the above results and~\eqref{lm4} provides
\begin{align}
\lefteqn{
   \bffh\sip\dd{\bfu}{}\Dsup + \bfn\sip\bfE\lpar(\bfu,p),(\bfuh,\hatp)\rpar\sip\bfth } & \suite
  = \dd{\ell}\fhat_s - \ell(\del{s}\theta_n)(\hatp_n\shp2\bftau\sip\del{s}\bfuh) - \lpar\bff\sip\del{s}\bfuh\rpar\theta_s
   + \Lpar \lpar f_s\bfn \shm p\bftau \rpar\sip\del{s}\bfuh + \kappa\ell \fhat_s - f_s\fhat_s \Rpar\theta_n
\end{align}
We next observe that, for given $\G$, $F=F(\bff,s)$ (the explicit dependence in $s$ stemming from the dependence in $\G$ of $F$) and derive
\begin{align}
  \lsqb F \shm \del{\bff}F \sip \bff \rsqb\,\DivS\bfth
 &= \lsqb F \shm \del{\bff}F \sip \bff \rsqb\del{s}\theta_s
 -\kappa\lsqb F \shm \del{\bff}F \sip \bff \rsqb\,\theta_n \\
 &= \text{d}_s\lpar \lsqb F \shm \del{\bff}F \sip \bff \rsqb\theta_s \rpar
 - \lpar\del{\bff}F\sip\del{s}\bff \shp \del{s} F \rpar\theta_s \suite
 + \lpar\del{\bff}F\sip\del{s}\bff\rpar\theta_s
 + \bff\sip\del{s}\lpar\del{\bff}F\rpar\theta_s
 - \kappa\lsqb F \shm \del{\bff}F \sip \bff \rsqb\,\theta_n \\
 &= \text{d}_s\lpar \lsqb F \shm \del{\bff}F \sip \bff \rsqb\theta_s \rpar
 - (\del{s} F)\theta_s + \lpar\bff\sip\del{s}\bfuh\rpar\theta_s
 - \kappa\lsqb F \shm \del{\bff}F \sip \bff \rsqb\,\theta_n
\end{align}
having recalled that $\bfuh\sheq\del{\bff}F$ on $\G$, used formula~\eqref{DivS:theta} for $\DivS\bfth$ and noticed that $\del{s} F\sheq\del{\bff}F\sip\del{s}\bff$ and $\text{d}_s F\sheq \del{\bff}F\sip\del{s}\bff + \del{s} F$. Finally, summing the last two equalities and rearranging terms yields Equation~\eqref{forward:wall}.
\end{document}